\documentclass[12pt,letterpaper]{amsart}
\usepackage[pass]{geometry}
\usepackage{amssymb,amsmath,amsthm,amsgen}
\usepackage{comment}
\usepackage{}

\newcommand{\old}[1]{}
\theoremstyle{plain}
\newtheorem{thm}{Theorem}[section]
\newtheorem{lem}[thm]{Lemma}
\newtheorem{conj}{Conjecture}
\newtheorem{cor}[thm]{Corollary}

\newtheorem{prop}[thm]{Proposition}
\theoremstyle{definition}

\newtheorem{remark}[thm]{Remark}

\newtheorem{qn}[thm]{Question}

\numberwithin{equation}{section}
\numberwithin{equation}{thm}

\def\la{{{\lambda}}}

 at 10truept

\title{The conjugacy action of $S_n$ and  modules induced from centralisers}

\author{Sheila Sundaram}
\address{Pierrepont School, One Sylvan Road North, Westport, CT 06880}
\email{shsund@comcast.net}
\date{29 October 2015; revised 26 September 2017}

\dedicatory{For  Benjamin and Priyanka}

\thanks{   }


\begin{document}

\begin{abstract}  We establish, for the character table of the symmetric group,  the positivity of  the row sums indexed by irreducible characters, when restricted to various subsets of the conjugacy classes.  A notable example is that  of partitions with all parts odd.  More generally, we study representations related to the conjugacy action of the symmetric group.  These arise as sums of  submodules induced from centraliser subgroups, and  their Frobenius characteristics have elegant descriptions, often as a multiplicity-free sum of power-sum symmetric functions.  We describe a general framework in which such representations, and consequently such linear combinations of power-sums, can be analysed. The conjugacy action for the symmetric group,  and more generally for a large class of groups, is known  to contain every irreducible. We   find other representations of dimension $n!$ with this property, including a twisted analogue of the conjugacy action. 


\end{abstract}
\maketitle

\section{Introduction}

This paper uses plethystic generating functions to provide a general framework for analysing  representations of the symmetric group which arise as sums of modules induced from centralisers. In particular we are led to the study of the conjugation action of $S_n$ on itself, and related permutation and sign-twisted permutation representations.  Further inspiration was provided  by working through  the invaluable  collection of problems and references in Richard Stanley's {\it Enumerative Combinatorics Vol. 2}, and especially  \cite[Exercise 7.71]{St4EC2}.  It follows from work of Solomon \cite{So}  that the Frobenius characteristic (\cite{M}, \cite{St4EC2}) of the $S_n$-action on itself by conjugation is a positive integer combination of the power-sum symmetric functions, each one occurring exactly once.  More recently it was shown, first in \cite{F}, and then in \cite{Sch}, that (for $n\neq 2$) this representation contains every irreducible, or equivalently, the characteristic is  a {\it positive} integer combination of every possible  Schur function.  For the symmetric group $S_n,$ this translates into the statement that in its character table, the row-sum indexed by any irreducible character is positive.  

The conjugacy action has been shown to contain every irreducible for a large class of groups, including the finite simple groups \cite{HSTZ}. This result is equivalent to a non-vanishing condition on the kernel of the adjoint action (as in Passman's problem \cite{HSTZ}).  Interestingly, the symmetric group $S_n$ seems to fall into a separate category from that studied in \cite{HSTZ}.   The conjugacy action decomposes into orbits indexed by centralisers; for $S_n$ these are symmetric powers of representations induced by the trivial character of a cyclic subroup, and hence our motivation for studying the symmetric and exterior powers of such representations.

 These investigations result in the discovery of several other $S_n$-modules with the same dimension as the regular representation, sharing properties of the conjugation action.  Many  seem to contain every irreducible.  In addition, the conjugacy action of the alternating subgroup is shown to play a prominent role in these modules.  Our methods  also provide a framework for the study of  many   nonnegative linear combinations of power-sums.

We write $\lambda\vdash n$ to mean that $\lambda$ is an integer partition $\lambda_1\geq \lambda_2\geq \ldots$ of $n.$ Write $\ell(\lambda)$ for the length, or number of parts, of $\lambda.$ Write $Par_n$ for the set of all partitions of $n,$ $DO_n$ for the subset of $Par_n$ consisting of partitions with all parts distinct and odd.  Fix $n$ and let $T_n$ denote any subset of the partitions of $n.$ 
Consider the linear combination of power-sum symmetric functions $P_{T_n}=\sum_{\lambda\in T_n} p_\lambda.$  We will say a symmetric function is  Schur-positive if it is a {\it nonnegative integer} combination of Schur functions.

Observe that $P_{T_n}$ is the Frobenius characteristic of a possibly virtual  $S_n$-module of nonzero dimension if and only if $T_n$ includes the partition $(1^n)$, in which case the dimension is $n!$.  Also note that the case  $T_n=Par_n$ corresponds to the result mentioned above about the conjugation action of $S_n$.  Sample  results of this paper follow. For the sake of completeness we include Solomon's result.

\begin{thm}
 For the following choices of $T_n, $ the sums $P_{T_n}$ are Schur-positive (and hence they are the Frobenius characteristic of some $S_n$-module).  

\begin{enumerate}
\item[(0)]  (Corollary 4.3)  $T_n=Par_n.$ 
\item (Theorem 4.6)  
$T_n=\{\lambda\in Par_n: \lambda \text{ has all parts odd}\}.$  
 \item (Theorem 4.11) $T_n=\{\lambda\in Par_n:  n-\ell(\lambda) \text{ even}\}.$ 
\item (Theorem 4.15) $T_n=\{\lambda\in Par_n: \lambda\notin DO_n, n-\ell(\lambda) \text{ even}\}, n\geq 2.$
\item (Theorem 4.11) $T_n=\{\lambda\in Par_n: \lambda\notin DO_n\},  n\geq 2.$ 
\item (Theorem 4.23) $T_n=\{\lambda\in Par_n: \lambda_i=1 \text{ or }k\},$ for any fixed $k\geq 1.$ 
\item (Theorem 5.6) $T_n=\{\lambda\in Par_n: \lambda_i \text{ divides }k\},$ for any fixed $k\geq 1.$  
\item (Theorem 5.6) For a fixed odd prime $p,$ 
 $T_n$  is the subset of partitions of $n$ with parts in the set $\{1,2,p,2p\}$ such that an even part occurs at most once.
\item (Theorem 5.6) 
$T_n$ consists of partitions $\lambda$ of $n$ such that odd parts are factors of $k,$   even parts occur at most once and $\lambda_i \text{  even }\Longrightarrow  (\lambda_i/2) |k \text{ but } \lambda_i \not| k,$  for any fixed $k\geq 1.$  
\end{enumerate}
\end{thm}
For  the character table of $S_n$ with rows indexed by the  irreducible characters, each of the above translates into a nonnegativity statement about a 
suitably {\it restricted} row-sum, omitting certain columns (i.e., conjugacy classes).  In some cases (e.g., (1)-(4) above) we are able to establish the stronger statement that these row sums are strictly positive, with an exception in case (4) for the row indexed by the sign character, which does not appear.  
We note also that by considering induced representations, the statement holds when each of the subsets described above is modified by restricting to partitions with at least $m$ parts equal to 1, for any fixed $m\geq 0.$

It is tempting to think that including the partition $(1^n)$ in the subset $T_n$ will guarantee Schur-positivity, by virtue of the dominating presence of the regular representation, but the following examples show that this is not the case.   Suppose $T_n$ consists of $(1^n)$ 
and all partitions $\lambda$ of $n$ such that $n-\ell(\lambda)$ is odd.  It is easily seen that  the sign representation occurs in $P_{T_n}$ with negative multiplicity $1-|\{\lambda\in Par_n: n-\ell(\lambda) \text{ odd}\}|$ when $n\geq 4.$  For instance, we have 
$P_{T_4}=p_1^4+p_2 p_1^2+p_4, P_{T_5}=p_1^5+p_2p_1^3+p_3p_2 +p_4p_1,$ and 
$P_{T_6}=p_1^6+p_2p_1^4+p_2^3+p_3p_2p_1
+p_4p_1^2+p_6,$ and in these the sign representation occurs with multiplicity $-1, -2, -4$ respectively.

      Other examples are  $T_6=\{(1^6), (2,1^4), (3^2),(4,2),(4,1^2)\}$ (the multiplicity of the Schur function indexed by $(2,1^4)$ is $-1$) and 
$T_6=\{(1^6), (2^3), (3,1^3), (4,1^2), (4,2)\}:$ in the latter the multiplicity of the Schur function indexed by $(3^2)$ is $-1$. 

The context in which these representations arise is interesting in its own right.  We define a {\it twisted} conjugacy action of  $S_n$  on itself, as the following  sum of induced modules from centralisers.  Recall that the centraliser of a permutation with $m_i$ cycles of length $i$ is a direct product of wreath product groups $\times_i S_{m_i}[C_i],$ where $C_i$ is a cyclic subgroup of order $i.$  First  define a representation of the wreath product $S_k[C_m]$ so that $S_m$ acts by conjugation on the $m$-cycles, and $S_k$ acts on the $k$ disjoint $m$-cycles according to the sign representation.  Then take outer tensor products, inducing up to $S_n$ from a centraliser of the form $\times_i S_{m_i}[C_i].$  We call the sum of all such modules, ranging over all possible conjugacy classes, the twisted conjugacy representation of $S_n.$ 
\begin{thm} The twisted conjugacy representation is  self-conjugate, has  dimension $n!,$ and contains every irreducible. Its Frobenius charactestic is given by $P_{T_n}$ where $T_n=\{\lambda \in Par_n: \lambda  \text{ has all parts odd}\}.$ 
\end{thm}

By considering  conjugacy actions (ordinary and twisted) on the subsets of even and odd permutations separately, we show 
\begin{thm} (Theorems 4.17 and 4.19) The ordinary (respectively twisted) conjugacy action on the set of even permutations contains every irreducible for $n\geq 4$ (respectively $n\geq5$) while both 
actions on the set of odd permutations contain every irreducible except the sign, for $n\geq 2.$  
\end{thm}
The arguments are technical and make use of  number-theoretic results on the distribution of primes.  In the case of the ordinary conjugacy action,  Frumkin's result follows as a corollary.

We also study related representations of the alternating group, and show the following:

\begin{thm} (Theorem 6.4) Consider the conjugacy action of the alternating group $A_n$  on itself, and  the representation of $S_n$ obtained by inducing this action up to $S_n.$ This is a self-conjugate representation of $S_n$ of dimension $n!,$ in which every irreducible occurs for $n\neq 3.$  It contains the conjugacy action on the even permutations as an invariant submodule.  
\end{thm}

As a corollary of the preceding theorem we obtain a different proof of a result that follows from work of Heide, Saxl, Tiep and Zalesski \cite{HSTZ}, namely, that the conjugacy action of $A_n$ ($n\neq 3$) on itself also contains every $A_n$-irreducible.  See Section 6.

The case $k=2$ of Theorem 1.1 (5) can be generalised in the following direction:
\begin{conj} Let $L_n$ denote the reverse lexicographic ordering \cite[p. 6]{M} on the set of partitions of $n.$ Then the sum $\sum p_\lambda,$ taken over any final segment of the total order $L_n,$
i.e., any interval of the form $[\mu, (1^n)]$ for fixed $\mu,$  (and thus  necessarily including  the partition $(1^n)$) is Schur-positive. 
\end{conj}

The paper is organised as follows.  In Sections 2 and 3 we describe a  general framework within which  these questions can be investigated.  In Section 4 we specialise these theorems to the conjugacy action of $S_n$, and in Section 5 we consider symmetric and exterior powers of an arbitrary linear character of a cyclic subgroup of $S_n$ (the Foulkes characters \cite{Fo}). In Section 6 we consider the alternating group, and more generally a finite group $G$ with a subgroup $H$ of index 2.  

At the end of the paper we have included four tables, generated using Stembridge's SF package for Maple.  Tables 1 and 2 contain respectively the decomposition into irreducibles for the conjugacy action and  its twisted counterpart, up to $n=10.$  Table 3 lists the decomposition into irreducibles of the conjugacy action on the set of even permutations and on the set of odd permutations, while Table 4 lists the  decomposition into irreducibles on the same two sets, but for the twisted conjugacy action.  Tables 3 and 4 contain this data up to $n=12.$

\section{The plethystic framework}
We follow \cite{M} and \cite{St4EC2} for notation regarding symmetric functions.  In particular, $h_n,$ $e_n$ and $p_n$ denote respectively the complete homogeneous, elementary and power-sum symmetric functions.  If $ch$ is the Frobenius characteristic map from the representation ring of the symmetric group $S_n$ to the ring of symmetric functions with real coefficients, then 
$h_n=ch(1_{S_n})$ is the Frobenius characteristic of the trivial representation, and $e_n=ch({\rm sgn}_{S_n})$ is the Frobenius characteristic of the sign representation of $S_n.$   If $\mu$ is a partition of $n$ then 
define $p_\mu =\prod_i p_{\mu_i};$ $h_\mu$ and $e_\mu$ are defined multiplicatively in analogous fashion.  As in \cite{M}, $s_\mu$ denotes the Schur function indexed by the partition $\mu;$ it is the Frobenius characteristic of the $S_n$-irreducible indexed by $\mu.$  Letting $z_\mu$ denote the size of the centraliser in $S_n$ of an element of cycle-type $\mu,$ we have that 
$z_\mu^{-1} p_\mu$ is the Frobenius characteristic of the class function whose value is 1 on the conjugacy class of type $\mu$ and zero on all other classes.  Finally, $\omega$ is the involution on the ring of symmetric functions which takes $h_n$ to $e_n,$ corresponding to tensoring with the sign representation.

If $q$ and $r$ are characteristics of representations of $S_m$ and $S_n$ respectively, they yield a representation of the wreath product $S_m[S_n]$ in a natural way, with the property that when this representation is induced up to $S_{mn},$ its Frobenius characteristic is the plethysm $q[r].$ For more background about this operation, see 
\cite{M}.  

In this paper we study representations of $S_n$ induced from centralisers.  The centraliser of an   element with cycle-decomposition corresponding to the partition $\mu$ of $n$ is a direct product of wreath products $\times_i S_{m_i}[C_i],$ where $\mu$ has $m_i$ parts equal to $i$  and $C_i$ is a cyclic subgroup of order $i$ (generated by an $i$-cycle). Note that the centraliser has cardinality $z_\mu=\prod_i i^{m_i} (m_i!).$  As indicated in the preceding paragraph, we will make extensive use  of properties of the  plethysm operation.  

Define \begin{align*} &H(t)=\sum_{i\geq 0}t^i  h_i, \quad E(t) = \sum_{i\geq 0} t^i  e_i; \\
 &H^{\pm}(t) = H(-t)=\sum_{i\geq 0} (-1)^{i} t^i h_i,\quad
  E^{\pm}(t) = E(-t)=\sum_{i\geq 0} (-1)^{i}t^i e_i ;\\
&H=\sum_{i\geq 0}  h_i, \quad E = \sum_{i\geq 0}  e_i; \qquad \qquad 
 H^{\pm} = \sum_{i\geq 0} (-1)^{i} h_i,\quad
  E^{\pm} = \sum_{i\geq 0} (-1)^{i} e_i .\\
\end{align*} 

Now let $\{q_i\}_{i\geq 1}$ be a sequence of symmetric functions, each $q_i$ homogeneous of degree $i.$  For example, each $q_i$ may be the Frobenius characteristic of a possibly virtual representation of $S_i.$  
Let $Q=\sum_{i\geq 1}q_i$ and $Q(t)=\sum_{n\geq 1} t^n q_n.$

For each partition $\lambda$ of $n\geq 1$  with $m_i(\lambda)=m_i$ parts equal to $i,$ let $|\lambda|=n = \sum_i i m_i$ be the size of $\lambda,$ and 
$\ell(\lambda)=\sum_i m_i(\lambda)=\sum_i m_i$ be the length (total number of parts) of $\lambda.$ 
Also we say a partition $\mu$ is contained in $\lambda,$ written 
$\mu\subseteq \lambda,$ if  $\mu_i\leq \lambda_i$ for all $i.$ 
  In this case $\lambda/\mu$ is a skew partition and $s_{\lambda/\mu}$ is the associated skew Schur function.  (When $\mu=\emptyset, $ $s_\lambda$ is the ordinary Schur function indexed by the partition $\lambda.$ )

Define, for each partition $\lambda$ of $n\geq 1$  with $m_i(\lambda)=m_i$ parts equal to $i,$
\begin{center} $H_\lambda[Q]=\prod_{i:m_i(\lambda)\geq 1} h_{m_i}[q_i]$ and $E_\lambda[Q]=\prod_{i:m_i(\lambda)\geq 1} e_{m_i}[q_i].$  \end{center} 

Also define \begin{center}$H^{\pm}_\lambda[Q]=\prod_{i:m_i(\lambda)\geq 1}(-1)^{m_i}h_{m_i}[q_i]$ and $E^{\pm}_\lambda[Q]=\prod_{i:m_i(\lambda)\geq 1} (-1)^{m_i}e_{m_i}[q_i].$  \end{center}

For the empty partition (of zero) we define $H_\emptyset [Q]=1=
E_\emptyset[Q]=H^{\pm}_\emptyset [Q]=
E^{\pm}_\emptyset[Q].$

 We have, by definition of $H^{\pm}$ (respectively $E^{\pm}$),  since $\sum_i m_i=\ell(\lambda),$ 
\begin{center}
$H^{\pm}_\lambda[Q]=(-1)^{\ell(\lambda)} H_\lambda[Q]$ 
and $E^{\pm}_\lambda[Q]=(-1)^{\ell(\lambda)} E_\lambda[Q].$ 
\end{center}

 We will derive some general facts about  the plethysms $H[Q]$ and $E[Q].$ Let $H[Q](t)$ (respectively $E[Q](t)$) denote the power series in $t,$ with coefficients in the ring of symmetric functions, such that the coefficient of $t^n$ is the degree $n$ term  (i.e., the symmetric function of homogeneous degree $n$) in $H[Q]$ (respectively $E[Q]$).  We define  $H^{\pm}[Q](t)$ and 
$E^{\pm}[Q](t)$ analogously.  With the convention that $Par$, the set of all partitions of nonnegative integers,
 includes the unique empty partition of zero,  by the preceding observations and standard properties of plethysm \cite{M} we have 

$$H(v)[Q](t)=\sum_{\lambda\in Par} t^{|\lambda|}v^{\ell(\lambda)} H_\lambda[Q], \qquad \text{and } E(v)[Q](t)=\sum_{\lambda\in Par} t^{|\lambda|} v^{\ell(\lambda)}E_\lambda[Q];$$  

$$ H^{\pm}[Q](t)=\sum_{\lambda\in Par} t^{|\lambda|} H^{\pm}_\lambda[Q]=   \sum_{\lambda\in Par} t^{|\lambda|} (-1)^{\ell(\lambda)} H_\lambda[Q],            
 $$
$$ \text{and } E^{\pm}[Q](t)=\sum_{\lambda\in Par} t^{|\lambda|} E^{\pm}_\lambda [Q] =\sum_{\lambda\in Par} t^{|\lambda|} (-1)^{\ell(\lambda)} E_\lambda[Q].$$

Also write $Q^{alt}$ for the alternating sum $\sum_{n\geq 1} (-1)^{i-1} q_i =  q_1-q_2+ q_3-\ldots.$  


We record some properties of plethysm that are important for our purposes (see \cite{M}):

\begin{prop} If $q, r$ are symmetric functions of homogeneous degrees $m$ and $n,$ respectively,  $g_1, g_2$ are arbitrary symmetric functions, and $\lambda$ is any partition, then 
\begin{enumerate}
\item $(fg)[q]=f[q]\cdot g[q];$
\item $s_\lambda [q+r]=\sum_{\mu\subseteq \lambda} s_{\lambda/\mu} [f] s_\mu[g];$ we single out the  special cases  $\lambda=(n), \lambda=(1^n):$
$h_n[q+r]=\sum_{k=0}^n h_k[q] h_{n-k}[r]$ and 
$e_n[q+r]=\sum_{k=0}^n e_k[q] e_{n-k}[r];$
\item $q[-r]= (-1)^{\text{deg }q} \cdot (\omega q)[r]=(-1)^m  (\omega q)[r];$
\item $\omega(q[r])=\left( \omega^{\text{deg }r}(q)\right)[\omega r]=\left( \omega^{n}(q)\right)[\omega r].$
\end{enumerate}
\end{prop}

\begin{lem} We have 
 
\begin{equation}
\omega(E_\lambda[Q]) = (-1)^{|\lambda|-\ell(\lambda)} H_\lambda[\omega(Q)^{alt}]=(-1)^{|\lambda|} \omega\left(E_\lambda^{\pm}[Q]\right) \end{equation}
\begin{equation*}=\sum_{\lambda\vdash n} \prod_i e_{m_{2i}}[\omega q_{2i}] h_{m_{2i+1}}[\omega q_{2i+1}];  \end{equation*}

\begin{equation} \omega(H_\lambda[Q]) = (-1)^{|\lambda|-\ell(\lambda)} E_\lambda[\omega(Q)^{alt}]=(-1)^{|\lambda|} \omega\left(H_\lambda^{\pm}[Q]\right)\end{equation}
 \begin{equation*}=\sum_{\lambda\vdash n} \prod_i h_{m_{2i}}[\omega q_{2i}] e_{m_{2i+1}}[\omega q_{2i+1}];  \end{equation*}
\end{lem}
\begin{proof}  We prove (2.2.1), since (2.2.2) is essentially the same. Observe that 
$|\lambda|-\ell(\lambda)=\sum_i (i-1)m_i \equiv \sum_j (2j)m_{2j+1}+\sum_j (2j-1) m_{2j}\equiv \sum_j m_{2j} \mod 2.$ Now 
a calculation using Proposition 2.1 shows that 

\begin{align*}
&\omega (E_\lambda[Q])  \\
&=  \prod_i \omega  (e_{m_{2i}} [ q_{2i}])  \cdot 
\omega (e_{m_{2i+1}} [ q_{2i+1}]) 
=   \prod_i \omega^{2i} (e_{m_{2i}})[\omega q_{2i}] \cdot \omega^{2i+1} (e_{m_{2i+1}})[\omega q_{2i+1}]  \\
&=  \prod_i e_{m_{2i}}[\omega q_{2i}] h_{m_{2i+1}}[\omega q_{2i+1}] 
=  \prod_i(-1)^{m_{2i}} h_{m_{2i}}[-\omega q_{2i}] h_{m_{2i+1}}[\omega q_{2i+1}] \\
&=(-1)^{|\lambda|-\ell(\lambda)} \prod_i h_{m_i}[(-1)^{i-1}\omega q_i]
=(-1)^{|\lambda|-\ell(\lambda)} H_\lambda[\omega(Q)^{alt}] \\
\end{align*}

Equation (2.2.2) follows {\it mutatis mutandis}; we need only interchange $e$'s and $h$'s in the above sequence.  The other equalities follow by applying the previous observations about $H_\lambda^{\pm}$ and $E_\lambda^{\pm}.$
\end{proof}

\begin{prop} Let $\{q_i\}$ be a sequence of symmetric functions with 
each $q_i$ of homogeneous degree $i,$ and let $S$ be any subset of 
the positive integers.  Let $\bar{S}$ denote the complement of $S.$ 
Let $Q=\sum_{i\geq 1} q_i$ as before, and for any subset $T$ of the positive integers, let $Q_T=\sum_{i\in T}  q_i.$
Then \begin{enumerate}
\item $H[Q_S]=E^{\pm}[Q_{\bar{S}}] \cdot H[Q];$
\item $E[Q_S]=H^{\pm}[Q_{\bar{S}}] \cdot E[Q].$
\end{enumerate}

In particular if each $q_i$ is the Frobenius characteristic of a true $S_i$-module, then the coefficient of $t^n$  in $E^{\pm}[Q_{\bar{ S}}]\cdot  H[Q]$ (respectively $H^{\pm}[Q_{\bar{S}}] \cdot E[Q]$) is also the characteristic of a true $S_n$-module, for any subset $S$ of the positive integers.

\end{prop}

\begin{proof} We only prove (1), since (2) is similar.
First note that, in the ring of symmetric functions, $H[Q]=H[Q_S]\cdot H[Q_{\bar{S}}].$  This follows, for example, by using the fact that $H[Q]=\sum_{\lambda\in Par} H_\lambda[Q],$ and keeping track of partitions $\lambda$ with parts 
$i\in S.$ Thus we have
$H[Q_S]\cdot H[Q_{\bar{S}}]= H[Q].$ 

Recall the well-known identity \cite{M}  $H(t)\cdot E^{\pm}(t)=1=E(t)\cdot H^{\pm}(t)$   (i.e., $H$ and $E^{\pm}$ are multiplicative inverses in the ring of symmetric functions, as are $E$ and $H^{\pm}.$) Thus we have 
$1/H[Q_{\bar{S}}]=(\frac{1}{H})[Q_{\bar{S}}]=E^{\pm}[Q_{\bar{S}}].$ 
The result follows immediately.   \end{proof}

The singleton case $\bar{S}=\{\alpha\}$ deserves special mention, since then we have 
$E^{\pm}[Q_{\bar{S}}]=\sum_{r\geq 0} (-1)^r e_r[q_{\alpha}],$ 
and similarly for $H^{\pm}.$   Then we have 
\begin{cor} $$H[Q_S]
=(\sum_{r\geq 0} (-1)^r e_r[ q_{\alpha}])\, H[Q];$$

$$E[Q_S]=(\sum_{r\geq 0} (-1)^r h_r[ q_{\alpha}])\, E[Q].$$
In particular if $Q$ and $S$ are such that $H_\lambda[Q_S]$ (respectively $E_\lambda[Q_S]$) is a representation for every $\lambda\vdash n$, then the same is true for the 
term of degree $n$ in the right-hand sides above.  
\end{cor}

Next we prove a result giving information about the irreducibles appearing in $H[Q_S]$ and $E[Q_S]$ in this very general setting.  We will apply a similar idea in Section 4 to prove that certain representations of the symmetric group share the property of the regular representation of containing at least one copy of each irreducible.   Curiously, we need a  number-theoretic result.  

\begin{thm}  \cite[Theorem 418]{HW}; \cite[Sec. 3.1, Exercises 28 and 30]{Ro}\begin{enumerate}
\item (Bertrand's Postulate): For every positive integer $n\geq 3,$ there is a prime $p$ such that $n/2<p<n.$ 
\item Every positive integer $n$ can be written as a sum of distinct primes and the integer 1.  Equivalently, for every $n$ there is a partition $\delta$ of $n$ whose parts are all distinct, and prime or equal to 1.  \end{enumerate}
\end{thm}

Bertrand's Postulate actually states that for each positive integer $m\geq 2,$ there is a prime $p$ such that $m<p<2m.$ 
If $n=2m,$ then $n\geq 4,$ and $n/2=m,$ $n=2m.$ 
If $n=2m+1,$ then $n\geq 5$ and we have $(n-1)/2=m<p<2m<n$ so that $m+1\leq p<n.$ But $n/2<m+1$ so we are done.  
If $n=3,$ clearly the  prime 2 satisfies $n/2<p <n.$

\begin{prop}  Let $\{q_i:i\geq 1\}$  be a sequence of Frobenius characteristics of $S_i$-modules (so $q_1=h_1$). Suppose that whenever $k$ is prime, the representation corresponding to $q_k$ contains every irreducible of $S_k.$  Let $\delta\vdash n$ be a partition with all parts prime and distinct.  Then the $S_n$-module whose characteristic is $H_\delta[Q]$ (respectively $E_\delta[Q]$) also contains every $S_n$-irreducible.

\end{prop}

\begin{proof} We have $H_{\delta}[Q]=\prod q_{\delta_i}$ where the 
$\delta_i$ are all distinct and each is either prime or equal to 1.  Thus each $\delta_i$ contains the sum of Schur functions 
$\sum_{\lambda\vdash i} s_{\lambda}.$  It suffices to establish the result when $\delta$ has only two parts, say $\delta_1=k, \delta_2=m$.   Let $\mu$ be a partition of $n,$ and choose any partition $\nu$ of $k.$ 
Consider the skew-Schur function $s_{\mu/\nu}.$ It is well-known \cite{M}  that this decomposes into irreducibles according to the 
formula $s_{\mu/\nu}=\sum_{\rho\vdash m} c^{\mu}_{\nu, \rho} s_{\rho},$ where the $c^{\mu}_{\nu, \rho}$ are the Littlewood-Richardson coefficients.  Since  the skew-Schur function $s_{\mu/\nu}$ is a nonzero symmetric function, ($k<n$), there is a partition $\rho\vdash m$ such that $c^{\mu}_{\nu, \rho}\neq 0.$  Equivalently, for every $\mu\vdash n,$ $s_{\mu}$ appears in the product $s_{\nu} \cdot s_{\rho},$ 
which in turn appears in the product $q_{\delta_1} q_{\delta_2}.$  
This finishes the proof.  \end{proof}

\section{A specific choice of $Q(t)$}

We will study the symmetric functions $H[Q]$ and $E[Q]$ for $Q$ of a specific form, which we now describe.  Let $\psi(n)$ be any real-valued function defined on the positive integers. 
Define symmetric functions $f_n$ by 
\begin{center} $f_n = \dfrac{1}{n} \sum_{d|n} \psi(d) p_d^{\frac{n}{d}},$\end{center}
so that \begin{center}  $\omega(f_n) =  \dfrac{1}{n} \sum_{d|n} \psi(d) (-1)^{n-\frac{n}{d}} p_d^{\frac{n}{d}}.$\end{center}
Note that,  when $\psi(1)$ is a positive integer,  this makes $f_n$ the Frobenius characteristic of a possibly virtual $S_n$-module whose dimension is $(n-1)!\psi(1).$

Also define the associated polynomial in one variable, $t,$ by
\begin{center}
$f_n(t) =\dfrac{1}{n} \sum_{d|n} \psi(d) t^{\frac{n}{d}}.$
\end{center}

\begin{prop}  Let $F(t)=\sum_{t\geq 1} t^i f_i, \quad  
(\omega F)^{alt}(t)=\sum_{i\geq 1} (-1)^{i-1} t^i\omega(f_i).  $  Then 
\begin{align} &F(t)= \log \prod_{d\geq 1} (1-t^d p_d)^{-\frac{\psi(d)}{d}}\\ 
&(\omega F)^{alt}(t)= \log \prod_{d\geq 1} (1+t^d p_d)^{\frac{\psi(d)}{d}}\end{align}
\end{prop}
\begin{proof} \begin{align*} F(t) &=\sum_{n\geq 1} {t^n \over n} \sum_{d|n} \psi(d) p_d^{n\over d}
=\sum_{m\geq 1} \sum_{d\geq 1} {t^{md}\over md} \psi(d) p_d^m \quad(\text{put}\ n=md)\ \\
&= \sum_{d\geq 1}{1\over d}\psi(d) \sum_{m\geq 1} {1\over m} (t^d p_d)^m
= \sum_{d\geq 1} {1\over d}\psi(d) \log(1-t^d p_d)^{-1}\\
&=\log \prod_{d\geq 1} (1-t^d p_d)^{-{\psi(d)\over d}}.
\end{align*}

Similarly, \begin{align*}(\omega F)^{alt}(t)&=\sum_{n\geq 1} (-1)^{n-1} t^n\omega(f_n)
=\sum_{n\geq 1} (-1)^{n-1}{t^n \over n}\sum_{d|n} \psi(d) (-1)^{n- \frac{n}{d}} p_d^{n\over d}\\
&= \sum_{m\geq 1} \sum_{d\geq 1} {t^{md}\over md}  \psi(d) {\bf(-1)^{m-1}} p_d^m \quad(\text{again\ put}\ n=md)\ \\
&=\sum_{d\geq 1} {1\over d}\psi(d) \sum_{m\geq 1} {(-1)^{m-1}\over m} (t^d p_d)^m
=  \sum_{d\geq 1} {1\over d}\psi(d) \log(1+t^d p_d)\\
&=\log \prod_{d\geq 1} (1+t^d p_d)^{\psi(d)\over d}.
\end{align*}
\end{proof}

\begin{thm} Let $H(v)=\sum_{n\geq 0} v^n h_n$ and 
$E(v)=\sum_{n\geq 0} v^n e_n.$  
We have the following plethystic generating functions:

 (Symmetric powers) 
\begin{equation}H(v)[F](t) = \sum_{\lambda\in Par} t^{\lambda|}  v^{\ell(\lambda)}H_\lambda[F]=\prod_{m\geq 1} (1-t^m p_m)^ {-f_m(v)}\end{equation}
(Exterior powers) 
\begin{equation}E(v)[F](t)=\sum_{\lambda\in Par} t^{\lambda|} v^{\ell(\lambda)}E_\lambda[F] =\prod_{m\geq 1} (1-t^m p_m)^{f_m(-v)}\end{equation}
 (Alternating exterior powers)
\begin{center}$\sum_{\lambda\in Par} t^{\lambda|} (-1)^{|\lambda|-\ell(\lambda)} v^{\ell(\lambda)}\omega(E_\lambda[F])$\end{center}
\begin{equation}=\sum_{\lambda \in Par} t^{\lambda|}v^{\ell(\lambda)}H_\lambda[\omega(F)^{alt}]
=H(v)[\omega(F)^{alt}](t)
= \prod_{m\geq 1} (1+t^m p_m)^{f_m(v)}\end{equation}
(Alternating symmetric powers) 
\begin{center}$\sum_{\lambda\in Par} t^{\lambda|} (-1)^{|\lambda|-\ell(\lambda)} v^{\ell(\lambda)}\omega(H_\lambda[F])$\end{center}
\begin{equation}=\sum_{\lambda \in Par} t^{\lambda|}v^{\ell(\lambda)}E_\lambda[\omega(F)^{alt}]
=E(v)[\omega(F)^{alt}](t)
= \prod_{m\geq 1} (1+t^m p_m)^{-f_m(-v)}\end{equation}
\end{thm}

\begin{proof} 
Since $ H(v)=\exp\sum_{i\geq 1} \frac{v^ip_i}{i},$  Equation (3.1.1) of Proposition 3.1 gives (taking $t=1$):

\begin{align*} H(v)[F] &=\exp\sum_{i\geq 1} \frac{v^ip_i}{i} 
\left[\sum_{d\geq 1} \frac{\psi(d)}{d} \log(1- p_d)^{-1}\right] \\
&=\exp \sum_{i\geq 1} \frac{v^i}{i} \left[\sum_{d\geq 1} \frac{\psi(d)}{d} \log(1- p_{id})^{-1}\right] \ 
(\text{using } p_i[p_j]= p_{ij})\\
&=\exp \sum_{m\geq 1} {1\over m} \sum_{d|m} v^{m/d}\psi(d) \log(1- p_m)^{-1} \ (\text{putting } m=id)\\
&= \exp \sum_{m\geq 1}\log (1- p_m)^{-\frac{1}{m}\sum_{d|m}\psi(d)v^{m/d}}\\
&=\exp\log \prod_{m\geq 1} (1- p_m)^{-\frac{1}{m}\sum_{d|m}\psi(d)v^{m/d}}=\prod_{m\geq 1} (1- p_m)^{-f_m(v)} .
\end{align*}
 Since $ E(v)=\exp\sum_{i\geq 1} (-1)^{i-1} \frac{v^ip_i}{i},$ exactly as before we obtain, by replacing $\psi(d)$ with $\psi(d) (-1)^{\frac{m}{d}-1},$  
\begin{align*} E(v)[F] &=\exp \sum_{m\geq 1} {1\over m}  \sum_{d|m} \psi(d) v^{m/d}{\bf (-1)^{\frac{m}{d}-1}}\log(1- p_m)^{-1}\\
&=\exp\log \prod_{m\geq 1} (1- p_m)^{\frac{1}{m}\sum_{d|m}\psi(d) v^{m/d}{\bf (-1)^\frac{m}{d}}}
=\prod_{m\geq 1} (1- p_m)^{f_m(-v)}.
\end{align*}
 As above, using (3.1.2) of Proposition 3.1, and  Equation (2.2.1), we have 
\begin{align*} H(v)[(\omega F)^{alt}] 
&=\exp\sum_{i\geq 1} \frac{v^ip_i}{i} \left[\sum_{d\geq 1} \frac{\psi(d)}{d} \log(1+ p_d)\right]\\
&=\exp \sum_{i\geq 1} v^i \sum_{d\geq 1} \frac{\psi(d)}{id} \log(1+ p_{id}) \\
&=\exp \sum_{m\geq 1} {1\over m} \sum_{d|m} \psi(d)v^{m/d} \log(1+ p_m)\\
&= \exp \sum_{m\geq 1}\log (1+ p_m)^{\frac{1}{m}\sum_{d|m}\psi(d)v^{m/d}}\\
&=\exp\log \prod_{m\geq 1} (1+ p_m)^{\frac{1}{m}\sum_{d|m}\psi(d)v^{m/d}}=\prod_{m\geq 1} (1+ p_m)^{f_m(v)}.
\end{align*}
\begin{align*} &E(v)[(\omega F)^{alt}]\\ 
&=\exp\sum_{i\geq 1} (-1)^{i-1}\frac{v^i p_i}{i} \left[\sum_{d\geq 1} \frac{\psi(d)}{d} \log(1+ p_d)\right]=\exp \sum_{i\geq 1} v^i  (-1)^{i-1}\sum_{d\geq 1} \frac{\psi(d)}{id} \log(1+ p_{id}) \\
&=\exp \sum_{m\geq 1} {1\over m} \sum_{d|m} (-1)^{\frac{m}{d}-1}\psi(d) v^{m/d} \log(1+ p_m)\\
&= \exp \sum_{m\geq 1}\log (1+ p_m)^{\frac{1}{m}\sum_{d|m}\psi(d) v^{m/d}(-1)^{\frac{m}{d}-1}}\\
&=\exp\log \prod_{m\geq 1} (1+ p_m)^{-f_m(-v)}
=\prod_{m\geq 1} (1+ p_m)^{-f_m(-v)}.
\end{align*}

The generating functions in the statement of the theorem now follow, noting that the power of $t$ keeps track of the homogeneous degree of the symmetric function on each side.
\end{proof}

From Theorem 3.2 it is clear that  the values of the 
  polynomial $f_n(u)=\frac{1}{n}\sum_{d|n} \psi(d) u^{\frac{n}{d}},$ evaluated at $u=\pm 1,$ appear prominently in the symmetric and exterior powers of $F=\sum_{n\geq 1} f_n.$     It turns out that, regardless of the choice of the function $\psi(d),$ for the particular choice of $f_n$ in this section, the two numbers $f_n(1)$ and $f_n(-1)$ determine each other via a simple formula. In fact the values of $f_m(1)$ completely determine the symmetric function $f_n.$ We record this observation in

\begin{prop}  With $f_n$ defined as above, we have
\begin{enumerate}
\item $f_{2m+1}(-1)=-f_{2m+1}(1)$ for all $m\geq 0,$ and 
$f_{2m}(-1)=f_m(1) -f_{2m}(1)$ for all $m\geq 1.$ 
\item $\psi(n)= \sum_{d|n} \mu(\frac{n}{d})\, d\, f_d(1),$ and hence the symmetric function $f_n$ is completely determined by the values $f_d(1), d|n.$
\end{enumerate}
\end{prop}
\begin{proof} For Part (1), we have 
$f_{2m+1}(-1)=\frac{1}{2m+1} \sum_{d|(2m+1)} \psi(d) (-1)^{\frac{2m+1}{d}}=-f_{2m+1}(1).$
For the second equality, we compute 
$$f_{2m}(-1)+f_{2m}(1)= \frac{1}{2m}\sum_{d| (2m)} \psi(d)\left((-1)^{\frac{2m}{d}}+(1)^{\frac{2m}{d}}\right)
=2\cdot  \frac{1}{2m}\sum_{\substack {d| (2m)\\ \frac{2m}{d} \text{ even}}}\psi(d).$$
Since $\frac{2m}{d}$ is even if and only if $d$ divides $m,$  we have 
$f_{2m}(-1)+f_{2m}(1)=\frac{1}{m}\sum_{d| m} \psi(d)=f_m(1),$ 
as claimed.  

For Part (2),  the expression for $\psi(n)$ is obtained from the equality $f_n(1)=\frac{1}{n} \sum_{d|n} \psi(d)$  
by M\"obius inversion. \end{proof}

Theorem 3.2 has the following important corollary, which we will later  use to determine families of  linear combinations of power-sums that are Schur-positive:

\begin{thm} Suppose the functions $f_n$ are Frobenius characteristics of true representations of $S_n,$ i.e., $f_n$ is Schur-positive.  Then the following are also true representations of $S_n:$ 

\begin{enumerate} 
\item $ H_\lambda[F]$ for each $\lambda\in Par;$
\item The coefficient of $t^n$ in 
$\prod_{m\geq 1} (1-t^m p_m)^ {-f_m(1)}=\sum_{\lambda\vdash n} H_\lambda[F];$
\item $ E_\lambda[F]$ for each $\lambda\in Par;$
\item The coefficient of $t^n$ in 
$\prod_{m\geq 1} (1-t^m p_m)^{f_m(-1)}=\sum_{\lambda\vdash n} E_\lambda[F].$
\item The coefficient of $t^n$ in 
$$\frac{1}{2}\left(\prod_m (1-t^m p_m)^{f_m(-1)} + \prod_m (1+(-1)^{m-1}t^m p_m)^{f_m(1)}\right)=  \sum_{\substack {\lambda\vdash n \\ n-\ell(\lambda) \text{ even}}} E_\lambda[F];$$ 
\item The coefficient of $t^n$ in 
$$\frac{1}{2}\left(\prod_m (1-t^m p_m)^{f_m(-1)} - \prod_m (1+(-1)^{m-1}t^m p_m)^{f_m(1)}\right)=  \sum_{\substack {\lambda\vdash n \\ n-\ell(\lambda) \text{ odd}}} E_\lambda[F]; $$ 
\item The coefficient of $t^n$ in 
$$\frac{1}{2}\left(\prod_m (1-t^m p_m)^{-f_m(1)} + \prod_m (1+(-1)^{m-1}t^m p_m)^{-f_m(-1)}\right)=  \sum_{\substack {\lambda\vdash n \\ n-\ell(\lambda) \text{ even}}} H_\lambda[F]; $$ 
\item  The coefficient of $t^n$ in 
$$\frac{1}{2}\left(\prod_m (1-t^m p_m)^{-f_m(1)} - \prod_m (1+(-1)^{m-1}t^m p_m)^{-f_m(-1)}\right)=  \sum_{\substack {\lambda\vdash n \\ n-\ell(\lambda) \text{ odd}}} H_\lambda[F]. $$ 
\end{enumerate}
\end{thm}

\begin{proof} We use Theorem 3.2 with the specialisation $v=1.$ Parts (1)-(4) are immediate from the definitions. Parts (5) and (6) follow from Theorem 3.2, Equations (3.2.2) and (3.2.3), by splitting up the sums according to sign, namely:
$$\sum_{\substack {\lambda\vdash n \\ |\lambda|-\ell(\lambda) \text{ even}}} t^{|\lambda|} E_\lambda[F] 
+\sum_{\substack{\lambda\vdash n\\ |\lambda|-\ell(\lambda) \text{ odd}}} t^{|\lambda|} E_\lambda[F] 
=\prod_{m\geq 1} (1-t^mp_m)^{f_m(-1)},\qquad (A)$$
and, applying the involution $\omega$ to the alternating exterior power in Equation (3.2.3):
$$\sum_{\substack{\lambda\vdash n \\ |\lambda|-\ell(\lambda) \text{ even}}} t^{|\lambda|} E_\lambda[F] 
-\sum_{\substack{\lambda\vdash n \\ |\lambda|-\ell(\lambda) \text{ odd}}} t^{|\lambda|} E_\lambda[F] 
=\prod_{m\geq 1} (1+(-1)^{m-1}t^mp_m)^{f_m(1)}.\qquad (B)$$
Adding and subtracting (A) and (B) immediately give (5) and (6) respectively.  

Parts (7) and (8) follow in identical fashion using Equation (3.2.1) and (3.2.4).  \end{proof}

Now assume $\psi(1)=1$ in the definition of $f_n.$  Then $f_n$ is the Frobenius characteristic of a possibly virtual $S_n$ module of dimension $(n-1)!$, whose restriction to $S_{n-1}$ is the regular representation.  This is easily seen by taking the partial derivative of $f_n$ with respect to the first power-sum symmetric function $p_1$ 
\cite{M}.  We will need the following fact about this differential operator \cite{M}:

 For symmetric functions $u$ and $v$ each of homogeneous degree,
\begin{equation}
\frac{\partial}{\partial p_1} (u[v])=  \left(\frac{\partial u}{\partial p_1}\right) [v] \cdot \frac{\partial v}{\partial p_1} 
\end{equation}

In what follows we will often use the Frobenius characteristic of an $S_n$-module to refer to the module itself.

\begin{prop}  Now assume $\psi(1)=1,$ so that the restriction of the representation $f_n$ to $S_{n-1}$ is the regular representation, i.e., 
$\frac{\partial f_n}{\partial p_1}=p_1^{n-1}.$ With notation as above, let $H[F](t)=G(t)=\sum_{n\geq 0} t^n g^{sym}_n$ and $E[F](t)=G^{ext}(t) = \sum_{n\geq 0} t^n g^{ext}_n.$
Then as virtual modules, we have 
$$\frac{\partial}{\partial p_1}g_{n+1}^{sym}=\frac{\partial}{\partial p_1}(p_1 g^{sym}_n)
= g_n^{sym} +p_1 \frac{\partial}{\partial p_1}( g^{sym}_n)$$ 
and 
$$ \frac{\partial}{\partial p_1}g_{n+1}^{ext}=\frac{\partial}{\partial p_1}(p_1 g^{ext}_n)=g_n^{ext} +p_1 \frac{\partial}{\partial p_1}( g^{ext}_n).$$

Equivalently, we have the identities
\begin{center} $\frac{\partial}{\partial p_1}g_{n+1}^{sym}
=\sum_{i=0}^n g_{n-i}^{sym} p_1^i,$ \qquad 
$\frac{\partial}{\partial p_1}g_{n+1}^{ext}
=\sum_{i=0}^n g_{n-i}^{ext} p_1^i.$ \end{center}
\end{prop}

\begin{proof} The hypothesis about $f_n$ can be restated in terms of symmetric functions as $\frac{\partial f_n}{\partial p_1}=p_1^{n-1}.$    We take partial derivatives with respect to $p_1$ in the equation $H[F](t)=G(t)$, first observing two key points: 

$ \frac{\partial h_n}{\partial p_1}=h_{n-1}$  implies that $\frac{\partial H}{\partial p_1}=H,$ and, because each $f_n$ restricts to the regular representation of $S_{n-1},$   $\frac{\partial F(t)}{\partial p_1}=\sum_{n\geq 1} t^n p_1^{n-1} =t(1-t p_1)^{-1}$

Hence we have, using the fact about restriction of induced wreath product modules:

\begin{center}
$ H[F](t)\cdot t(1-tp_1)^{-1} = \frac{\partial G(t)}{\partial p_1},$ \qquad (A)
\end{center}
and hence 
\begin{center} 
 $t\cdot G(t)=(1-tp_1)\frac{\partial G(t)}{\partial p_1}. \qquad (B)$
\end{center}
But $\frac{\partial G(t)}{\partial p_1}=\sum_{n\geq1} t^n \frac{\partial g_n^{sym}}{\partial p_1}
=t\cdot \sum_{n\geq 0} t^n \frac{\partial g_{n+1}^{sym}}{\partial p_1}.\qquad (C)$

Thus the first equation (A)  becomes
\begin{center} $G(t)\cdot t(1-tp_1)^{-1}=\frac{\partial G(t)}{\partial p_1} \qquad (A1)$ \end{center} 
and the second equation (B) becomes 
\begin{center}$ t G(t)= (1-tp_1) \cdot \frac{\partial G(t)}{\partial p_1} \qquad (B1)$ \end{center}

 From (B1) and (C) it follows upon equating coefficients of $t^n,$ $n\geq 0,$ that 
$$g_n^{sym}=\frac{\partial g_{n+1}^{sym}}{\partial p_1} -p_1 \frac{\partial g_{n}^{sym}}{\partial p_1},$$ 
while  (A1) and (C) yield 
$\frac{\partial}{\partial p_1}g_{n+1}^{sym}
=\sum_{i=0}^n g_{n-i}^{sym} p_1^i.$ 

The statement for $g_n^{ext}$ is completely analogous.\end{proof}

\section{Conjugation in the symmetric group}

In this section we let $\psi(d)=\phi(d),$ where $\phi$ is the number-theoretic totient  function.  Then $f_n=\frac{1}{n}\sum_{d|n} \phi(d) p_d^{\frac{n}{d}}$ is the Frobenius characteristic of  the representation $1\uparrow_{C_n}^{S_n}$ of $S_n$ induced from the trivial representation of a cyclic subgroup of order $n,$ that is,  of $S_n$ acting on the class of $n$-cycles by conjugation. 

\begin{lem}  We have, for $f_n=\frac{1}{n}\sum_{d|n} \phi(d) p_d^{\frac{n}{d}}:$
\begin{enumerate}
\item 
$f_n(1)=1$ for all $n\geq 1.$ 
\item $f_n(-1)=\frac{1}{n}\sum_{d|n}\phi(d)(-1)^{\frac{n}{d}}=(-1)\delta_{n,\text{odd}}.$ 
\end{enumerate}
\end{lem}
\begin{proof} 

\begin{enumerate} 
\item This is immediate from the  well-known fact that $\sum_{d|n} \phi(d)=n.$ 
\item First let $n$ be odd.  Then by Lemma 3.3, $ f_n(-1) = (-1) f_n(1)$ and the result follows from (1).  Now let $n$ be even.  
Then Lemma 3.3 tells us that $f_n(-1)= f_{n/2}(1)-f_n(1)=1-1=0$ by (1). 
\end{enumerate}
This establishes the claim. \end{proof}

Define $Par^{odd}_n$ (respectively $Par^{even}_n$) to be the set of partitions of $n$ with all parts odd (respectively even), and  $Par^{\neq}_n$ to be the set of partitions of $n$ with distinct parts.  (An old result of Euler states that $Par^{odd}_n$ and $Par^{\neq}_n$ have the same cardinality.)  As in the Introduction, define $DO_n=Par^{\neq}_n\cap Par^{odd}_n,$ the set of all partitions of $n$ into parts that are odd and distinct.  (It is also well-known that 
$DO_n$ has the same cardinality as the set of self-conjugate partitions of $n.$ )
From Theorem 3.2 (taking $v=1$) and Lemma 4.1, we have 

\begin{thm} Let $f_n$ be the Frobenius characteristic of $1\uparrow_{C_n}^{S_n}$  as in the preceding lemma.  Then 
\begin{enumerate}
\item (Symmetric powers) 
\begin{equation}
 H[F](t) =\sum_\mu t^{|\mu|} H_\mu[F] 
= \prod_{m\geq 1} (1-t^m p_m)^ {-1}=\sum_{n\geq 0} t^n \sum_{\lambda\in Par_n} p_\lambda;\end{equation}
\begin{equation}\sum_{\mu\vdash n}H_\mu[F]= \sum_{\lambda\in Par_n} p_\lambda.\end{equation}
\item (Exterior powers) 
\begin{equation}E[F] (t)=\sum_\mu t^{|\mu|} E_\mu[F]
=\prod_{m\geq 0} (1-t^{2m+1} p_{2m+1})^{-1}=\sum_{n\geq 0} t^n \sum_{\lambda\in Par^{odd}_n} p_\lambda.\end{equation}
\begin{equation}\sum_{\mu\vdash n}E_\mu[F]= \sum_{\lambda\in Par_n^{odd}} p_\lambda.
\end{equation}
\item (Alternating exterior powers) 
\begin{equation}\sum_{\lambda} (-1)^{|\lambda|-\ell(\lambda)}t^{|\lambda|}\omega (E_\lambda[F])
 = \prod_{m\geq 1} (1+t^m p_m)=\sum_{n\geq 0} t^n \sum_{\lambda\in Par^{\neq}_n} p_\lambda.\end{equation}
\begin{equation}\sum_{\mu\vdash n}(-1)^{|\mu|-\ell(\mu)} \omega(E_\mu[F])= \sum_{\lambda\in Par_n^{\neq}} p_\lambda.\end{equation}
\item (Alternating symmetric powers) 
\begin{equation}\sum_{\lambda} (-1)^{|\lambda|-\ell(\lambda)}t^{|\lambda|}\omega (H_\lambda[F])
= \prod_{m\geq 0} (1+t^{2m+1} p_{2m+1})=\sum_{n\geq 0} t^n \sum_{\lambda\in DO_n} p_\lambda.\end{equation}
\begin{equation}\sum_{\mu\vdash n}(-1)^{|\mu|-\ell(\mu)}\omega (H_\mu[F]) = \sum_{\lambda\in DO_n} p_\lambda.\end{equation}
\end{enumerate}
\end{thm}

We now interpret the preceding calculations  in the context of 
 the operation of $S_n$ on itself by conjugation.  This action results in a permutation representation $\psi(S_n)$ whose orbits are the conjugacy classes of $S_n$ (indexed by partitions of $n$); since the centraliser of any element of type $\lambda$ is a direct product of wreath products of the form 
$\times_i S_{m_i}[C_i]$ if $\lambda $ has $m_i$ parts equal to $i,$ it follows easily that the degree $n$ term in $H[F]$ is the Frobenius characteristic of the conjugacy action of $S_n$ on itself, and more specifically, the degree $n$ term in $H_\lambda[F]$ is the characteristic of the conjugation action on the conjugacy class indexed by the partition $\lambda.$ Hence we obtain, as an immediate corollary to (1) above, the following result, Part (2) of which is the case $G=S_n$ of \cite{So}.   
\begin{cor}  Let $f_n=ch(1\uparrow_{C_n}^{S_n}).$   Let $\psi(S_n)$ denote the representation of $S_n$ acting on itself by conjugation.  Then 
\begin{enumerate} 
\item $ch(\psi(S_n))=\sum_{\lambda\vdash n}H_\lambda[F]=\sum_{\lambda\vdash n, \lambda =\prod_i i^{m_i}} \prod_i h_{m_i}[f_i].$

\item  \cite[Exercise 7.71 (c)]{St4EC2} $ch(\psi(S_n))=\sum_{\lambda\vdash n} p_\lambda.$

\end{enumerate}
In particular, if $\chi^\nu$ is the irreducible $S_n$-character indexed by the partition $\nu$ of $n,$ we recover the known result (\cite{So})  that for every partition $\nu$ of $n,$
$\sum_{\lambda\vdash n} \chi^\nu(\lambda)$ is a nonnegative integer.
\end{cor}

\begin{proof}  These facts are immediate from Theorem 3.2 (1), since 
the coefficient of $t^n$ in $\prod_{m\geq 1} (1-t^m p_m)^ {-1}$ is 
precisely $ \sum_{\lambda\vdash n, \lambda =\prod_i i^{m_i}} \prod_i h_{m_i}[f_i].$  The proof is completed by using the expansion 
$p_\lambda =\sum_{\nu\vdash n} s_\nu \chi^\nu(\lambda).$   
\end{proof}
 
The conjugation action of $S_n$ on itself has been the subject of several papers ( \cite{F},  \cite{Sch}), and  (\cite[8.6.3]{R}).  See also \cite[Solution to Exercise 7.71 (a) and (b))]{St4EC2} for a more modern treatment in terms of symmetric functions.  For arbitrary finite groups we have the following two results,  of which the first, due to Frame, curiously pre-dates the second considerably.

\begin{thm}  Let $G$ be a finite group. The permutation representation of $G$ acting on itself by conjugation has the following properties.  
\begin{enumerate} 
\item \cite{Frm} Its character is given by the formula $\sum_{\chi\in Irr(G)} \chi\bar\chi,$ where $Irr(G)$ is the set of irreducible characters of $G.$
\item  \cite{So} The multiplicity of the irreducible $\chi$ equals the sum $\sum_C \chi(C)$ where $C$ ranges over all the conjugacy classes of $G.$  
\end{enumerate}
\end{thm}

By considering  the conjugation action on specific centralisers,  T. Scharf  \cite{Sch} gave another proof of a result originally due to A. Frumkin \cite{F} that for $n\neq 2,$ every irreducible appears in the conjugation action of $S_n.$   

Frumkin's result is 
\begin{thm} (\cite{F}, \cite{Sch}) If $n\neq 2,$ every irreducible of $S_n$ appears in $\psi(S_n)$.  Equivalently, the character table row-sums 
$\sum_{\lambda\vdash n} \chi^\nu(\lambda)$ are positive integers  for every partition $\nu$ of $n\neq 2.$ 
\end{thm}

Now consider the exterior powers in Theorem 4.2:  Let $\varepsilon(S_n)$ be the representation of $S_n$ whose Frobenius characteristic is the coefficient of $t^n$ in $E[F(t)]$ for $f_n=1\uparrow_{C_n}^{S_n}.$ Thus  $$ch(\varepsilon(S_n))=\sum_{\lambda \in Par_n}E_\lambda[F]=\sum_{\lambda\vdash n, \lambda =\prod_i i^{m_i}} \prod_i e_{m_i}[f_i].$$ 
Equivalently, $\varepsilon(S_n)$
 is a sum of transitive monomial  representations, with orbits indexed by partitions $\lambda$ of $n$, and stabiliser equal to the centraliser of type $\lambda$.  The action of $S_n$ is trivial within each cycle, and permutes cycles of the same length according to the sign representation.  Thus  $\varepsilon(S_n)$ may be considered to be the  conjugation action twisted by the sign.  Note that the second statement of the next result, namely, that $\varepsilon(S_n)$ is self-conjugate, is not obvious from the above definition:  the submodules with  Frobenius characteristic  $E_\lambda[F]$ are not themselves self-conjugate for arbitrary $\lambda.$

\begin{thm} We have, with $f_n=ch(1\uparrow_{C_n}^{S_n}):$ 
\begin{enumerate}
\item $ch(\varepsilon(S_n))=E[F]=\sum_{\lambda\in Par^{odd}_n} p_\lambda.$
\item $\varepsilon(S_n)$ is a self-conjugate representation of $S_n$ of degree $n!$
\item If $\chi^\nu$ is the irreducible $S_n$-character indexed by the partition $\nu$ of $n,$ then for every $\nu\vdash n,$ 
$\sum_{\lambda\in Par^{odd}_n} \chi^\nu(\lambda)$ is a nonnegative integer.

\end{enumerate}
\end{thm}
\begin{proof} 
Part (1) follows from (2) of Theorem 3.2 (with $v=1$), by taking the coefficient of $t^n$ in the right-hand side of (2). 
Part (2) is now immediate, since $\omega(p_i)=(-1)^{i-1} p_i$ and thus 
$\omega$ fixes $p_i$ if $i$ is odd.  Hence if $\lambda$ has all parts odd, $p_\lambda$ is fixed by $\omega.$  The statement about the degree is clear since only $p_{(1^n)}=p_1^n$ contributes to the degree.  Part (3) follows upon writing the expression in (1) in terms of Schur functions.    \end{proof}

Computations with Stembridge's symmetric functions package SF for Maple suggest that, just as in the case of the conjugation action, every irreducible occurs in $\varepsilon(S_n).$  In order to prove this we will need a lemma about the irreducibles appearing in the representations $1\uparrow_{C_k}^{S_k}$ :

\begin{lem} If $k$ is an odd prime, the representation $f_k=ch(1\uparrow_{C_k}^{S_k})$ contains every irreducible 
except for those indexed by $(k-1,1)$ and $(2, 1^{k-2}).$   More generally, $f_n$ always contains the trivial representation, and never contains the representation $(n-1, 1)$, and contains the representation $(1^n)$ (once) if and only if $n$ is odd and $(2, 1^{n-2})$ (once) if and only if $n$ is even.

\end{lem}

\begin{proof}Note that $f_2=h_2.$ 
 We have $f_k=\frac{1}{k} (p_1^k+(k-1)p_k)$ since $\phi(k)=(k-1)$ when $k$ is prime.  Let $\nu\vdash k.$ 

Let $\chi^\nu$ be the irreducible indexed by $\nu;$ its degree   $f^\nu$ is the number of standard Young tableaux of shape $\nu.$ Expanding in terms of Schur functions, we have 
$f_k=\frac{1}{k} (\sum_\nu f^\nu s_\nu +(k-1) \sum_\nu \chi^\nu((k))) s_\nu,$  where we have used the fact that 
$p_\mu=\sum_\lambda \chi^\lambda(\mu) s_\lambda.$ 

But for the $k$-cycle $(k),$ it is well-known (\cite{M}, \cite{St4EC2}) that $\chi^\nu((k))$ is nonzero if and only if $\nu=(k-r, 1^r)$ is a hook shape, in which case it equals $(-1)^r.$ Hence the multiplicity of the irreducible $\nu$ in $f_k$ is $\frac{1}{k}f^\nu$ if $\nu$ is not a hook, 
and equals $\frac{1}{k}( {k-1 \choose r} +(k-1) (-1)^r)$ if 
$\nu=(k-r, 1^r), 0\leq r\leq k-1.$   It follows that 
the trivial and sign representations ($r=0$ and $r=k-1$ respectively) each occur exactly once, the representation $(k-1,1)$ never occurs, and, because $r=k-2$ is odd, 
the representation $(2, 1^{k-2})$ never occurs.

If $\nu\vdash k$ is any partition other than a hook, the multiplicity is given by $\frac{1}{k} f^\nu.$  Since this number is necessarily an integer, and $f^\nu$ is a positive integer (being the number of standard Young tableaux of shape $\nu,$ or the degree of the irreducible indexed by $\nu$), it follows that it is nonzero.  

The last statement for general $n$ follows from a theorem of Kr\'askiewicz and Weyman, and independently Stanley, (\cite{KW}, \cite[Solution to Exercise 7.88 (b)]{St4EC2}; see also \cite[Cor. 8.10]{R}) that the multiplicity of the irreducible indexed by $\lambda$ in $f_n$ equals the number of standard Young tableau with major index congruent to 0 modulo $n.$   \end{proof}

\begin{remark} Maple computations with SF in fact support the conjecture that $f_n$ contains every irreducible besides $(n-1,1), (2,1^{n-2})$ and $(1^n).$
 (The author is grateful to R. Stanley for verifying this up to $n=29.$)   This conjecture was recently settled in the affirmative by Josh Swanson \cite{Sw}.
\end{remark}

\begin{thm} The representation $\varepsilon(S_n)$ (as well as the representation $\psi(S_n)$ for $n\neq 2$) contains every irreducible representation of $S_n.$   In particular, the nonnegative integers 
$\sum_{\lambda\in Par_n^{odd}} \chi^\nu(\lambda)$ are in fact positive for all partitions $\nu$ of $n.$ 
\end{thm}

\begin{proof}  Our starting point is the formula $ch(\varepsilon(S_n))=\sum_{\lambda \in Par_n} E_\lambda[F]$  
(recall the definition of $E_\lambda[H]$ from the preliminaries in Section 2).
We will show that for every partition $\lambda\in Par,$ there is a partition $\tau(\lambda)$ of $n=|\lambda|$ such that 
the Schur function $s_\lambda$ appears in the representation $E_{\tau(\lambda)}[H],$  i.e., such that $\langle s_\lambda, E_{\tau(\lambda)}[H] \rangle>0.$

First note that $f_i=h_i$ for $i=1,2.$ Also the trivial representation $h_n=s_{(n)}$ appears in $f_n=E_{(n)}[F],$ and the sign representation also appears in $\varepsilon(S_n)$ because  $e_n=e_n[f_1]=E_{(1^n)}[F].$ Hence we can take $\tau((n))=(n)$ and $\tau((1^n))=(1^n).$  Similarly since 
$s_{(n-1,1)}$ appears in the product $h_{n-1} h_1,$ which in turn appears in $f_{n-1}f_1 =e_1[f_{n-1}] e_1[f_1]=E_{(n-1,1)}[F]$ for $n\geq 3,$ we can take 
$\tau((n-1,1))=(n-1,1)$ for $n\geq 3.$ Likewise, $s_{(2, 1^{n-2})} $ appears in $e_1[f_2] e_{n-2}[f_1] =E_{(2, 1^{n-2})}[F],$ so that we may take 
$\tau((2,1^{n-2}))=(2,1^{n-2})$ when $n\geq 3.$

We proceed by induction on $n=|\lambda|.$ The cases $n=0$ and $n=1$ are trivial.  Let $n=2.$ Then for $\lambda\vdash 2,$ $\tau(\lambda)=\lambda,$ since $h_2=s_{(2)}$ appears in $f_2=e_1[f_2]=E_{(2)}[F],$ and $e_2=s_{(1^2)}$ appears in 
$e_2=e_2[f_1]=E_{(1^2)}[F].$ 

The three partitions of $n=3$ have already been addressed. Now let $n=4.$ Here the first new partition is $(2^2)$.  But it is easily checked that there is a standard Young tableau of shape $(2^2)$ with major index congruent to 0 modulo 4, and hence  $s_{(2^2)}$ appears in $f_4=e_1[f_4]=E_{(4)}[F].$ We may therefore take $\tau((2^2))=(4).$

For $n=5,$ considering only the partitions we have not already addressed, since 5 is an odd prime, we have that $\tau(\lambda)=\lambda$ for every partition $\lambda$ except $(4,1)$ and its conjugate $(2,1^3).$

By induction hypothesis, assume that for every $i\leq n,$ every partition  $\lambda$ of $ i$ admits a partition  $\tau(\lambda)$ of $ i$ such that 
$s_\lambda$ appears in $E_{\tau(\lambda)}[F].$  Let $n\geq 6.$ 

From the preceding discussion, we may assume $\lambda$ and $ \lambda^t$ do not equal $(n-1,1).$ By Bertrand's Postulate 
(see Theorem 2.5) there is a largest prime $q\geq 5$ such that $n/2<q<n,$ i.e.,  $n>q>n-q>1.$ We claim that there is  a partition $\mu$ of $ q$ such that $\mu\subset \lambda$ and 
neither $\mu$ nor its transpose is of the form $(m-1,1).$  Invoking Lemma 4.7, this guarantees that $s_\mu$ appears in $f_q=E_{(q)}[F].$

{\bf Case 1:} If $\mu$ is of the form $(m-1,1),$ then since $n-q>n/2\geq 1, $ and we have assumed $\lambda\neq (n-1,1),$ we know that all the squares in the skew shape $\lambda/\mu$  cannot be in row 1;  there is at least one square in the skew-shape  that is in row 2 or row 3. Hence we may modify our choice by taking $\mu=(q-2,2)$ or $\mu=(q-2,1^2).$ 
Both shapes satisfy our criteria (the second because $q\neq 4$) and hence the claim.

{\bf Case 2:} If $\mu^t$ is of the form $(m-1,1),$ then again there must be a square in the skew-shape $\lambda/\mu$ that is in row 2 or row 1; all the squares cannot be at the bottom because that would force 
$\lambda^t=(n-1,1).$ Thus we may modify our choice by taking $\mu=(2^2, 1^{q-4})$ or $\mu=(3,1^{q-3}).$ 
Both shapes satisfy our criteria (the second because $q\geq 5$) and  the claim follows.

Next consider  the set of partitions $\nu$ of $ n-q$ such that $s_\lambda$ appears in the product $s_\mu s_\nu,$ and choose such a $\nu.$ Assume first that 
{\it neither $\nu$ nor $\nu^t$ is of the form $(m-1,1),\ m\geq 2.$} Then by induction hypothesis, there is a partition $\tau(\nu)$ of $ n-q$ such that 
$s_\nu$ appears in $E_{\tau(\nu)}[F].$ Also $s_\mu$ appears in 
$E_{(q)}[F]=f_q$ by Lemma 4.7 since $q$ is an odd prime.  Hence 
$s_\lambda$ appears in $E_{\tau(\nu)}[F]\cdot  E_{(q)}[F]=E_{\tau(\nu)}[F] f_q.$ But this is just $E_{\tau(\nu)\cup (q)}[F],$ since $q>n-q$ implies that $q$ is distinct from the parts of $\tau(\nu).$ We have shown that we may take $\tau(\lambda)=\tau(\nu)\cup (q).$   (If $m=1,$ $\tau(\nu)$ is just $(1)$and $\tau(\lambda)$ is the two-part partition $(q,1).$)

Now assume {\it $\nu^t$ is of the form $(m-1,1), m\geq 2,$ i.e., $\nu=(2, 1^{m-2}).$} Then $s_\nu$ appears in the product $h_2 e_{m-2}[f_1]=e_1[f_2] \cdot e_{m-2}[f_1].$ It follows that $s_\lambda$ appears in $E_{(q)}[F]\cdot E_{(2)}[F]\cdot E_{(1^{m-2})}[F],$ so that we may take $\tau(\lambda)=(q, 2, 1^{m-2}).$  Since $q$ is an odd prime, we have $E_{(q)}[F]\cdot E_{(2)}[F]\cdot E_{(1^{m-2})}[F]=E_{(q, 2, 1^{m-2})}[F].$

Finally assume {\it $\nu$ is of the form $(m-1,1).$} Clearly $s_\nu$ appears in $f_{m-1}f_1$ where $m=n-q.$ Again we have two cases.

{\bf  Case (1a):} If $m-1=n-q-1>1,$ we have that $s_\lambda$ appears in 
$s_\mu s_\nu$ and thus in 
$f_q f_{m-1} f_1=E_{(q, m-1,1)},$ so we may take $\tau(\lambda)$ to be the partition 
$(q, m-1, 1)$ and this has distinct parts since $q$ is an odd prime which is greater than $n/2$ and $2\leq m-1=n-q-1<n/2-1.$

{\bf  Case (2a):} If $m=2,$ i.e., $\nu=(1^2),$ then $s_\nu=e_2$ appears in $e_2[f_1]=E_{(1^2)}[F],$ and thus we may take $\tau(\lambda)=(q, 1^2),$ so that $s_\lambda$ appears in $s_\mu e_2,$ which in turn appears in $f_q e_2=E_{(q)}[F]\cdot E_{(1^2)}[F]=E_{(q,1^2)}.$ 

This completes the induction, showing that every irreducible appears in $\varepsilon(S_n).$  \end{proof}

Tables 1 and 2 contain the decomposition into irreducibles up to $n=10$ for the ordinary and twisted conjugacy representations $\psi(S_n)$ and $\varepsilon(S_n).$  In these tables, the rows index the partitions of $n$ listed in reverse lexicographic order.

If $W_n$ is a representation of $S_n$ of dimension $n!,$ and $\rho_n$ denotes the regular representation of $S_n,$ then the induced representation 
$W_k\otimes \rho_{n-k}\uparrow_{S_k\times S_{n-k}}^{S_n}$ also has dimension $n!;$ if $W_n$ contains every $S_k$-irreducible, then the induced representation clearly also contains every $S_n$-irreducible.  Hence we have 
\begin{cor}  Fix $k\geq 0.$ Then the induced representations whose characteristics are 
$ch\, \psi(S_k)\cdot p_1^{n-k}, k\neq 2$ and $ch\, \varepsilon(S_k)\cdot p_1^{n-k}$ are both $n!$-dimensional representations of $S_n$ containing every irreducible; the second representation is self-conjugate.
\end{cor}
The proof of Theorem 4.9 can be adapted to give yet another proof that $\psi(S_n)$ contains every irreducible.  The remainder of this section is devoted to  proving a stronger result.  We start by introducing a new class of representations, 
motivated primarily by the fact that the $S_n$-action on itself by conjugation breaks up into invariant submodules indexed by partitions $\lambda$ of $n.$ 

For any subset $T$ of partitions of $n,$ define the representation $\psi(S_n,T)$ to be the conjugation action of $S_n$ on permutations whose types lie in $T,$ and similarly the representation $\varepsilon(S_n,T)$ to be the twisted conjugation action of $S_n$ on permutations whose types lie in $T.$  The Frobenius characteristics are thus the Schur-positive symmetric functions $ch(\psi(S_n,T))= \sum_{\lambda\in T} H_\lambda[F]$  
and $ch(\varepsilon(S_n,T))= \sum_{\lambda\in T} E_\lambda[F].$ 
 In particular, taking $T$ as a set to be the alternating subgroup $A_n$ or its complement, which we denote by $\bar A_n$, we have, 
since the sign of a permutation $\sigma\in S_n$ is given by 
$(-1)^{n-\ell(\mu)}$ if the type of $\sigma$ is the partition $\mu,$ we have 
\begin{center}$ ch(\psi(S_n,A_n))= \sum_{\substack{\lambda\in Par_n\\ n-\ell(\lambda) \text{ even}}} H_\lambda[F],\quad \text{and }
 ch(\psi(S_n,  \bar A_n))= \sum_{\substack{\lambda\in Par_n\\ n-\ell(\lambda) \text{ odd}}} H_\lambda[F];$\end{center}
\begin{center}$ ch(\varepsilon(S_n,A_n))= \sum_{\substack{\lambda\in Par_n\\ n-\ell(\lambda) \text{ even}}} E_\lambda[F],\quad \text{and }
 ch(\varepsilon(S_n,  \bar A_n))= \sum_{\substack{\lambda\in Par_n\\ n-\ell(\lambda) \text{ odd}}} E_\lambda[F].$\end{center}

\begin{thm} We have the following formulas for the Frobenius characteristics, and each  linear combination of 
power-sums below is Schur-positive: 
\begin{align} 
ch\, \psi(S_n) =ch\,\psi(S_n, A_n) +ch\,\psi(S_n,   \bar A_n)=\sum_{\mu \vdash n} H_\mu[F]=\sum_{\lambda\in Par_n}p_\lambda;\\
ch\, \psi(S_n,A_n)= \sum_{{\mu \vdash n}\atop {n-\ell(\mu) \text{ even}}} H_\mu[F]=
\frac{1}{2} \left( \sum_{\lambda\in DO_n}p_\lambda +\sum_{\lambda\in Par_n} p_\lambda\right);\\
ch\,\psi(S_n,  \bar A_n)= \sum_{\substack{\mu \vdash n\\ n-\ell(\mu) \text{ odd}}} H_\mu[F]=
\frac{1}{2} \sum_{\lambda\notin DO_n}p_\lambda ;\\
\frac{1}{2}\left(ch\, \psi(S_n)+\omega(ch\,\psi(S_n))  \right)=\frac{1}{2}\sum_{\mu \vdash n} \left(H_\mu[F]+ \omega(H_\mu[F])\right)=\sum_{\substack{\lambda\in Par_n\\ n-\ell(\lambda) \text{ even}}} p_\lambda;\\
ch\,\varepsilon(S_n) =ch\,\varepsilon(S_n, A_n) +ch\,\varepsilon(S_n,   \bar A_n)=\sum_{\mu \vdash n} E_\mu[F]=\sum_{\lambda\in Par_n^{odd}}p_\lambda;\\
 \omega(ch\,\varepsilon(S_n, A_n))=\sum_{\substack{\mu \vdash n\\ n-\ell(\mu) \text{ even}}}\omega( E_\mu[F])=\frac{1}{2}\left( \sum_{\lambda\in Par_n^{odd}}  p_\lambda +\sum_{\lambda\in Par_n^{\neq}}  p_\lambda  \right);\\
\omega(ch\,\varepsilon(S_n,   \bar A_n))=\sum_{\substack{\mu \vdash n \\ n-\ell(\mu) \text{ odd}}}\omega( E_\mu[F])= \frac{1}{2}\left( \sum_{\lambda\in Par_n^{odd}}  p_\lambda -\sum_{\lambda\in Par_n^{\neq}}  p_\lambda \right).
\end{align}
\end{thm} 
\begin{proof} Apply Theorem 4.2 and Theorem 3.4. The only point that requires comment is the presence of the fraction $\frac{1}{2}.$  In each case, one side is clearly Schur-positive; equating multiplicities of each Schur function, we have that twice the multiplicity on one side 
is a nonnegative integer, and hence the multiplicity itself must be a nonnegative integer.  Here we invoke the fact that the character values of $S_n$ are all integers. \end{proof}

Note that $\sum_{\lambda\in Par_n^{\neq}}  p_\lambda$ is of course NOT Schur-positive, having dimension zero. Even including the regular representation does not result in Schur-positivity, since for $n=7$ we have $\langle \sum_{\lambda\in Par_n^{\neq}}  p_\lambda, s_{(1^7)}\rangle=-2.$

\begin{cor} For $n\neq 2,$ the Schur-positive symmetric function $\sum_{\substack{\lambda\in Par_n\\ n-\ell(\lambda) \text{ even}}} p_\lambda $ contains $s_\nu$ for every partition $\nu$ of $ n.$
\end{cor}

\begin{proof}  We use  equations  (4.11.1) and (4.11.4) in Theorem 4.11.  Since for $n\neq 2$ the conjugacy representation contains every irreducible at least once, it is immediate from  Equation (4.11.4) that 
$\langle \sum_{\substack{\lambda\in Par_n\\ n-\ell(\lambda) \text{ even}}} p_\lambda, s_\nu\rangle >0.$ Hence, invoking again the fact that the characters of $S_n$ are integer-valued, i.e., 
every power-sum is an integer combination of Schur functions, 
the statement follows.  
\end{proof}

 Define the possibly virtual modules 
\begin{center}
$U_n^{DO}=\sum_{\lambda\in DO_n} p_\lambda, \quad$ 
$U_n^{+} =\sum_{\substack{\mu\vdash n, \mu  \notin DO_n\\ n-\ell(\mu) even}}
p_\mu,$\end{center}  and 
\begin{center} $U_n^{-}  = \quad \sum_{\substack{\mu\vdash n, \mu  \notin DO_n\\ n-\ell(\mu) odd}}
p_\mu=\quad \sum_ {\substack{\mu\vdash n\\     
  n-\ell(\mu) odd }} p_\mu   .$\end{center}

Note that $U_n^{+}$ is the characteristic of a module that  has dimension $n!$ (because $p_1^n$ is a summand), while $U_n^{-}$ and $U_n^{DO}$  are characteristics of  zero-dimensional (hence  virtual) $S_n$-modules.  

\begin{prop}  We have
\begin{align}
&  U_n^{DO}=\sum_{\lambda\vdash n} (-1)^{n-\ell(\lambda)} H_\lambda[F], \text{ and } U_n^{DO} \text{ is self-conjugate}\\
& U_n^+ \text{ is self-conjugate, while } \omega(U_n^-)=- U_n^-. \\
&  \psi(S_n)=U_n^+ + U_n^- +U_n^{DO}  \\
&  \omega \psi(S_n)=U_n^+ - U_n^- +U_n^{DO}\\
&  \psi(S_n,A_n)= \frac{1}{2}(U_n^+ + U_n^-) +U_n^{DO} \\
&   \psi(S_n,   \bar A_n)
= \frac{1}{2}(U_n^+ + U_n^-)  \\
&   2U_n^{-}=\psi(S_n) - \omega(\psi(S_n))  \\
&    U_n^{DO}=\psi(S_n, A_n)-\psi(S_n,   \bar A_n)  \\
&    U_n^{+} =\psi(S_n,   \bar A_n) + \omega (\psi(S_n,   \bar A_n))  \\
&   U_n^+  +U_n^{DO}= \psi(S_n, A_n)+ \omega (\psi(S_n,   \bar A_n))=\frac{1}{2}\left(\psi(S_n)+\omega(\psi(S_n))\right)  
\end{align}
\end{prop}

\begin{proof}  The first equation, Equation  (4.13.1), follows directly from the definition of $U_n^{DO}$ and Theorem 4.2, Part (4).  Now $\omega(p_\mu)= (-1)^{n-\ell(\mu)} p_\mu.$ 
Since  $n-\ell(\mu)$ is odd for all the partitions $\mu$ appearing in the decomposition of $U_n^{-}$,  we compute $\omega(U_n^{-}) = (-1) U_n^{-}.$  Also $n-\ell(\mu)$ is even if $\mu\in DO_n.$  Hence statement (4.13.2).

Equation (4.13.3) follows from Theorem 4.2, Part (1), since $U_n^+ + U_n^- +U_n^{DO}$ is the sum of all $p_\lambda$ as $\lambda$ ranges over all partitions of $n.$

Equation (4.13.4) now follows from the preceding paragraphs.  Equations (4.13.5) and (4.13.6) are consequences of Theorem 4.10, Parts (2) and (3) respectively, since 
$U_n^+ + U_n^- =\sum_{\lambda\notin DO_n} p_\lambda.$

Equation (4.13.7) follows by subtracting (4.13.4) from (4.13.3), while (4.13.8) follows by subtracting (4.13.6) from (4.13.5).  Equation (4.13.9) follows directly from equation (4.13.6) by applying the involution $\omega$ and using (4.13.2).  Finally (4.13.10) is immediate upon combining  (4.13.9) and (4.13.8). \end{proof}

\begin{cor} Let $\lambda$ be a partition of $ n.$ If $\lambda$ is not self-conjugate, 
then $\langle ch\, \psi(S_n), s_\lambda\rangle $ and $\langle ch\, \psi(S_n), s_{\lambda^t}\rangle $ are either both odd or both even.  
If $\lambda=\lambda^t,$ then the irreducible indexed by $\lambda$ does not appear in $U_n^{-}.$ 

\end{cor}

\begin{proof}  From equations (4.13.3) and (4.13.4) above, it follows that if $\lambda$ is not self-conjugate, then 
\begin{center}
$\langle ch\,\psi(S_n), s_\lambda\rangle +\langle ch\,\psi(S_n), s_{\lambda^t}\rangle=\langle ch\,\psi(S_n), s_\lambda\rangle +\langle \omega(ch\,\psi(S_n)), s_{\lambda}\rangle=2\langle U_n^+ +U_n^{DO}, s_\lambda\rangle $\end{center} 
 is even, as is 
$$\langle ch\,\psi(S_n), s_\lambda\rangle -\langle ch\,\psi(S_n), s_{\lambda^t}\rangle=\langle ch\,\psi(S_n), s_\lambda\rangle -\langle \omega(ch\,\psi(S_n)), s_{\lambda}\rangle =2\langle U_n^{-}, s_\lambda\rangle.$$

The claims follow.\end{proof}

An immediate consequence of  Proposition 4.13 is 
\begin{thm} $U_n^{+}$ and $U_n^+  + U_n^{DO}$ are both  Frobenius characteristics of   self-conjugate, true $S_n$-modules of dimension $n!$  We have 
\begin{align}
U_n^+&=\sum_{\substack{\mu\notin DO_n\\ n-\ell(\mu) \text{ even}}} p_\mu =\sum_{\substack{\mu\vdash n\\ n-\ell(\mu) odd}} (H_\mu[F] + \omega(H_\mu[F])) \notag \\
&=ch\,\psi(S_n,   \bar A_n) + \omega (ch\,\psi(S_n,   \bar A_n)). \\
U_n^+  +U_n^{DO}&= \sum_{n-\ell(\mu) \text{ even}} p_\mu \notag \\
&= ch\,\psi(S_n, A_n) + \omega (ch\,\psi(S_n,   \bar A_n))\\
&= \frac{1}{2}\left(ch\,\psi(S_n)+\omega(ch\,\psi(S_n))\right).\\
\text{Hence we also have } \notag\\
 U_n^+  +U_n^{DO} 
&=\omega(ch\,\psi(S_n, A_n)) + ch\,\psi(S_n,   \bar A_n).
\end{align}

\end{thm}

\begin{proof} This follows directly from Proposition 4.13, 
Equations (4.13.9) and (4.13.10), recalling that 
$ch\,\psi(S_n,   \bar A_n) =\sum_{\substack{\mu\vdash n\\ n-\ell(\mu) odd}} H_\mu[F].$  The second equality in (4.15.2) was established in (4.11.4)  (see Theorem 4.11). 

Now recalling that $ch\,\psi(S_n)=ch\,\psi(S_n,A_n) +ch\,\psi(S_n,   \bar A_n)$ (see (4.11.1)),  we obtain (4.15.3).  This is also a consequence of the fact that $U_n^+ +U_n^{DO}$ is self-conjugate. \end{proof} 

We now need a number-theoretic lemma on the distribution of primes that follows from a result of Nagura \cite{JitsuroN}.

\begin{lem} If $n\geq 12$ then there is an odd prime $q$ such that 
  $n>q>n-q\geq 5.$  If $n\geq 11$ or $n=9,$ there is an odd prime $q$ such that  $q>n-q\geq 4.$ 
\end{lem}

\begin{proof} Nagura's theorem \cite{JitsuroN} states that, for every positive integer $m\geq 25, $ there is at least one prime between $m$ and $6m/5,$ i.e., such that $m<p<6m/5.$   By iterating this, as observed in \cite{SR}, we find that
\begin{center} $m<p_1<6m/5<p_2<36m/25< p_3<216m/125<2m.$ \end{center} 
Hence, for $m\geq 25,$  there are three distinct primes such that 
$m<p_1<p_2<p_3<2m.$  Putting $n=2m+1, $ this gives
$(n-1)/2=m <p_1<p_2<p_3<2m<n,$ so that $n/2=m+1/2<m+1\leq p_1<p_2<p_3<2m<n,$ while putting $n=2m$ leads to the same conclusion.  It follows that for $n\geq 50,$ there are primes $p_1,p_2,p_3$ such that 
\begin{center} $n/2<p_1<p_2< p_3<n.$ \end{center} 

Setting $q=p_1,$ we have $n-q=(n-p_3)+(p_3-p_2)+(p_2-p_1)\geq 1+2+2=5,$ for $n\geq {\bf 50}.$ 

The remaining cases, viz.,  $12\leq n< 50,$ are easily checked by hand, down to $n=11:$ 
the prime 7 is less than 11 and greater than 11/2, and $11-7=4.$ 
Similarly  $q= 5$ satisfies the requirement for $n=9.$ Also, if $50>n\geq 12,$ one checks that there is a prime $q,$ $n>q>n/2,$ such that $n-q\geq 5.$\end{proof}

Our next theorem strengthens the fact that the conjugacy action contains every $S_n$-irreducible \cite[see Theorem 4.5]{F}.

\begin{thm} Every $S_n$-irreducible appears in the conjugacy action on the subset of even permutations,  for $n\geq 4.$  Also every $S_n$-irreducible 
except the sign,  appears in the conjugacy action on the set of odd permutations, for $n\geq 2.$ 

Equivalently:
For every partition $\lambda\vdash n,$ the Schur function $s_\lambda$ appears 
\begin{enumerate} 
\item in the (Schur-positive) symmetric function $ch\, \psi(S_n, A_n)=
\sum_{\substack{\mu\vdash n\\ n-\ell(\mu) even}} H_\mu[F]$  for all $n\geq 4;$
\item in the (Schur-positive) symmetric function 
$ch\,\psi(S_n,   \bar A_n) =\sum_{\substack{\mu\vdash n\\ n-\ell(\mu) odd}} H_\mu[F],$ except for $\lambda=(1^n),$ which never appears when $n\geq 2.$
\end{enumerate}
\end{thm}
\begin{proof} Note that the sign representation will appear in 
$H_\mu[F]$ if and only if $\mu$ has distinct odd parts, since $f_n$ contains the sign if and only if $n$ is odd.  But $n-\ell(\mu)$ is necessarily even for such a $\mu.$ Hence the sign representation appears only in $\psi(S_n,A_n).$ 

 We will use an induction argument as in the proof of  Theorem 4.9.  However, because the sign representation does not appear in $\psi(S_n,   \bar A_n)$, we will need  Lemma 4.16.

Maple computations with SF show that the statement of the theorem holds for $n\leq 12.$  The resulting decompositions into irreducibles of $\psi(S_n,A_n)$ and $\psi(S_n, \bar A_n)$  are recorded in 
Table 3,  where we have denoted by $\lambda$  the $S_n$-irreducible indexed by the partition $\lambda.$

Our induction hypothesis will  be that for fixed $n\geq 12,$ whenever $4\leq k< n,$
and {\bf k}$\neq $ {\bf 2,3}, for every partition $\lambda$ of $ k,$ $\lambda\neq (1^k),$ there are partitions $\tau(\lambda)$ and $\sigma(\lambda)$ of $k$ such that 
$\ell( \tau(\lambda))$ is even, $\ell( \sigma(\lambda))$ is odd, and 
$s_\lambda$ appears in both $H_{\tau(\lambda)}[F] $ and  $H_{\sigma(\lambda)}[F].$  
For $\lambda=(1^k)$ we know there is a partition $\tau(\lambda)$ such that $k-\tau(\lambda)$ is even and 
$s_\lambda$ appears in  $H_{\tau(\lambda)}[F].$

It is convenient to address the extreme cases first.  We have
\begin{enumerate} 
\item Let $\lambda=(n).$ Then  we may always take $\tau(\lambda)=(n)$  because $s_{(n)}$ appears in $H_{(n)}[F]=f_{n}$ which contains  $h_n.$ Similarly we may take  $\sigma(\lambda)$ to be $(n-1,1),$ because if  $H_{(n-1,1)}[F]=f_{n-1} f_1$ if $n\geq 3,$ contains $h_{n-1}h_1,$ which in turn contains $h_n.$ This includes the case $n=2,$ for then we have 
$\sigma(\lambda)=(1^2)$ and $H_{(1^2)}[F]=h_2[f_1]=h_2.$ 
When $n=3,$ the two partitions $(2,1)$ and $(1^3)$ satisfy the required conditions:$H_{\sigma(\lambda)}[F]= f_2 f_1 $ or $h_3[f_1].$  We have shown that $s_{(n)}$ appears in $ch\,\psi(S_n, A_n)$ and in $ch\,\psi(S_n,   \bar A_n)$ for all $n.$
\item Let $\lambda=(1^n).$  Take $\tau(\lambda)=(n)$ if $n$ is odd, 
so that $H_\tau[F]=f_n,$  and $\tau(\lambda)=(n-1,1)$ if $n$ is even, $n\neq 2,$ so that $H_\tau[F]=f_{n-1}f_1,$ .  In each case, since $f_m$ contains the sign representation for odd $m,$  $s_{(1^n)}$ appears in 
$H_\tau[F];$ if $n$ is odd, $n-\ell(\tau)=n-1$ is even, and if $n$ is even, $n-\ell(\tau)=n-2$ is again even.  Hence $s_{(1^n)}$ appears in $H_\tau[F], n\neq 2,$ with even $n-\ell(\tau),$ i.e.,  in $ch\,\psi(S_n, A_n), n\neq 2.$
\item Let $\lambda=(n-1,1).$  Provided $n\geq 4,$ we may take $\tau(\lambda)$ to be $(n-1,1)$ and $\sigma(\lambda)$ to be $(n-2,1^2).$ In the first case $H_{(n-1,1)}[F]=f_{n-1}f_1$ which contains $h_{n-1} h_1$ and hence contains $s_{(n-1,1)};$ in the second case (since $n-2>1$) $H_{(n-2,1^2)}[F]=f_{n-2}h_2$  contains $h_{n-2} h_2,$ which in turn contains $s_{(n-1,1)}.$  
If $n=3$ we can take $\tau=(2,1),$ and then $s_{(2,1)}$ appears in 
$H_\tau[F]=f_2 f_1=h_2 h_1,$ i.e., in  $ch\, \psi(S_3, S_3\backslash A_3)$ but not in  $ch\,\psi(S_3,  A_3).$
Hence $s_{(n-1,1)}$ appears in $ch\,\psi(S_n, A_n)$ for $n\geq 4$ and in 
$ch\,\psi(S_n,   \bar A_n)$ for $n\geq 3.$
\item Let $\lambda=(2, 1^{n-2}).$  Since $f_r$ contains the sign representation if and only if $r$ is odd, it follows that 
\begin{enumerate}
\item If $n$ is odd, we may take $\tau(\lambda)=(n-2,2)$ and $\sigma(\tau)=(n-2,1^2), n\geq 4$ (the case $n=3$ has been addressed above);
\item If $n$ is even, we may take $\tau(\lambda)=(n-1,1)$ and $\sigma(\lambda)=(n)$ for $n\geq 2; $  because by Lemma 4.7, $s_{(2,1^{n-2})}$ appears in $H_{(n)}[F]=f_n$ whenever $n$ is even, in which case $n-\ell((n))$ is odd.
\end{enumerate}
Hence $s_{(2, 1^{n-2})}$ appears in  $ch\,\psi(S_n,  A_n)$ if $n\geq 4,$ and in  $ch\,\psi(S_n,   \bar A_n)$ if $n\geq 3.$
\end{enumerate}

Now let $n\geq 12,$ and let $\lambda$ be any partition of $n$ other than   $(1^n), (n-1,1)$ and $(2, 1^{n-2}).$ By Lemma 4.16 there is  a  prime $q$  such that $q>n-q\geq 5.$ We claim that because $\lambda,\lambda^t \neq (n-1,1),$ we can choose a partition  $\mu \subset \lambda$ such that $\mu, \mu^t\neq (|\mu|-1,1).$ Hence, invoking Lemma 4.7, this guarantees that $s_\mu$ appears in $f_q=H_{(q)}[F]$ for the odd prime $q.$
\begin{enumerate}
\item[{\bf Case (a):}] If $\mu$ is of the form $(m-1,1),$ then since $n-q>n/2\geq 1, $ and we have assumed $\lambda\neq (n-1,1),$ we know that all the squares in the skew shape $\lambda/\mu$  cannot be in row 1;  there is at least one square in the skew-shape  that is in row 2 or row 3. Hence we may modify our choice by taking $\mu=(q-2,2)$ or $\mu=(q-2,1^2),$  noting that $q-2\geq 3.$
Both shapes satisfy our criteria (the second because $q\neq 4$) and hence the claim.
\item[{\bf Case (b):}] If $\mu^t$ is of the form $(m-1,1),$ then again there must be a square in the skew-shape $\lambda/\mu$ that is in row 2 or row 1; all the squares cannot be at the bottom because that would force 
$\lambda^t=(n-1,1).$ Thus we may modify our choice by taking $\mu=(2^2, 1^{q-4})$ or $\mu=(3,1^{q-3}).$ 
Both shapes satisfy our criteria (the second because $q\geq 5$) and  the claim follows.
\end{enumerate}

Now consider the set of partitions $\nu$ of $n-q$ such that $s_\lambda$ appears in the product $s_\mu s_\nu,$ and choose such a partition $\nu.$ 
\begin{enumerate}
\item[{\bf Case 1:}] Assume first that neither $\nu$ nor $\nu^t$ is of the form $(m-1,1),$ and $\nu\neq (1^m).$ By induction hypothesis there are partitions $\tau(\nu)$ and $\sigma(\nu)$ of even and odd length respectively such that 
$s_\nu$ appears in $H_{\tau(\nu)}[F]$ and $H_{\sigma(\nu)}[F].$ 
Hence $s_\lambda$ appears in $H_{(q)\cup \tau(\nu)}[F]$ and in $H_{(q)\cup \sigma(\nu)}[F].$  (Again $q>(n-2)/2$ implies that $q$ is distinct from the parts of $\tau(\nu)$ and $\sigma(\nu).$ )  Since the partitions $(q)\cup \tau(\nu)$ and $(q)\cup \sigma(\nu)$ have lengths of opposite parity, we are done.
\item[{\bf Case 2:}] Now suppose $\nu^t=(m-1,1),$ i.e., $\nu=(2, 1^{m-2}).$  From (4) above, since $m=n-q\geq 5,$  we can find partitions $\tau(\nu)$ and $\sigma(\nu)$ with lengths of opposite parity such that $s_\nu$ appears in both $H_{\tau(\nu)}[F]$ and $H_{\sigma(\nu)}[F].$  Then as before, 
since $s_\mu$ appears in $H_{(q)}[F],$ it follows that 
$s_\lambda$ appears in $H_{(q)\cup \tau(\nu)}[F]$ and $H_{(q)\cup\sigma(\nu)}[F],$  and since one part distinct from all the others, namely $q>n-q,$ was added to both 
partitions, again we have partitions with lengths of opposite parity, as claimed.
\item[{\bf Case 3:}]
Here we consider the case $\nu=(m-1,1)=(n-q-1,1).$  Since  $m=n-q\geq 5,$ from (3) above we have two partitions $\tau$ and $\sigma$ of opposite parity such that $H_\tau[F]$ and $H_\sigma[F]$ contain $s_\nu.$
\item[{\bf Case 4:}] Suppose $\nu=(1^m).$ Then $s_\nu$ appears  in $H_\tau[F]$ but only for some $\tau$ such that $|\tau|-\ell(\tau)=m-\ell(\tau)$ is even.   The argument is therefore more involved.  Note first  that if $s_\lambda$ appears in $s_\mu\cdot s_{(1^m)}$ for some $\mu$ of $q,$  it must also appear in the product  
$s_{\mu'}\cdot s_{(2,1^{m-2})}$  for some partition $\mu'$ of $q.$ (Recall that $m=n-q\geq 5.$)  
This follows by observing that $(1^m)\subset\lambda,$ and recalling that $\lambda
\neq (1^n).$   Since $f_q$ contains all partitions of $q$ {\bf except} $(q-1,1)$ and $(2, 1^{q-2}),$ the discussion in Case {\bf 2} and the preceding cases implies that it suffices to consider the following two sub-cases for the partition $\mu'$:
\begin{enumerate}
\item[{\bf Case 4-1:}] Assume $\mu'=(q-1,1).$ In this case we know from (3) above that $s_{\mu'}$ appears in $H_{(q-1,1)}[F]=f_{q-1}f_1,$ in $H_{(q-2,1^2)}[F]=f_{q-2}h_2[f_1],$  as well as in $H_{(q-2,2)}[F]=f_{q-2}f_2,$ the latter because $q-2>2$ and $f_{q-2}$ contains $h_{q-2}.$ 

Now let  $m$ be {\it even}. Then (4(b)) implies that $s_{(2, 1^{m-2})}$ appears in $H_{(m)}[F]=f_m$ and hence, taking $\tau(\lambda)=(q-2, 1^2)\cup (m),$ since $q-2\neq m$ ($q$ is odd and $m$ is even), we have an even-length partition (into four parts) such that $H_{\tau(\lambda)}[F]=f_{q-2}f_m h_2[f_1]$ contains $s_{\mu'}s_{(2,1^{m-2})}$ and hence contains $s_\lambda.$  Similarly taking $\sigma(\lambda)=(q-2,2)\cup (m)$ produces an odd-length partition (into three distinct parts: $q-2\neq m$ because $q$ is odd and $m$ is even) such that $H_{\sigma(\lambda)}[F]=f_{q-2}f_2 f_m$ contains $s_\lambda.$   

Next suppose $m$ is {\it odd}. Then from (4(a)) we know that $s_{(2, 1^{m-2})}$ appears in $H_{(m-2,2)}[F]=f_{m-2}f_2$ and $H_{(m-2,1^2)}[F]=f_{m-2}h_2[f_1].$  Taking $\tau=(q-2,1^2)\cup (m-2,2),$ and noting that $q-2> m-2>2$, (here we are using the hypothesis about $q$ that $m=n-q\geq 5$) we see that $H_\tau[F]$ contains $s_\lambda$ and $\tau$ has odd length.  Similarly noting that $q-1>m-1> m-2>2,$ we see that the even-length partition $\sigma=(q-1,1)\cup (m-2,2)$ has the same property:   $H_\sigma[F]$ contains $s_\lambda.$   This settles the case $\mu'=(q-1,1).$
\item[{\bf Case 4-2:}]  Assume $\mu'=(2, 1^{q-2}).$ 

First let $m$ be {\it even}. Invoking (4(b)), we see again (note that $q$ is odd and $m$ is even, so $q-2\neq m$) that the partitions $\tau=(q-2,2)\cup (m)$ and $\sigma=(q-2,1^2)\cup (m)$ have opposite parity and both $H_\tau[F]=f_{q-2}f_2f_m, H_\sigma[F]=f_{q-2} h_2[f_1]$ contain $s_\mu' s_{(2,1^{m-2}) }$ and hence contain $s_\lambda.$ 

Now let $m$ be {\it odd}. Then we can take $\tau=(q-2,2)\cup (m-2,1^2)$ to obtain an odd-length partition such that $H_\tau[F]=f_{q-2} f_2 f_{m-2} h_2[f_1]$  contains $s_{\mu'} s_{(2, 1^{m-2})}.$ Observe also that since $q-1$ is even,by Lemma 4.7,  $s_{\mu'}=s_{(2,1^{q-2})}$ appears in $f_{q-1} h_1=H_{(q-1,1)}.$ Hence since $q-1>m-1>m-2,$ we may in addition take $\sigma=(q-1,1)\cup (m-2,2)$ 
for the even length partition such that $H_\sigma[F]=f_{q-1}f_1 f_{m-2} f_2$ contains $s_{\mu'} s_{(2, 1^{m-2})}.$
\end{enumerate}

\end{enumerate}
 This finishes the inductive step, and hence the proof is complete.  \end{proof}

Combining the preceding result with (4.15.1) of Theorem 4.15, we immediately obtain
\begin{cor} As an $S_n$-module, $U_n^{+} $   contains every irreducible of $S_n.$  
Equivalently, the partial row sums $\sum_{\substack{\mu\notin DO_n \\ n-\ell(\mu) even}} \chi^\lambda(\mu)$  in the character table of $S_n$ are positive integers.
\end{cor}

We also have the analogous strengthening of Theorem 4.9:

\begin{thm} Every $S_n$-irreducible appears in the twisted conjugacy action on the subset of even permutations,  for $n\geq 5.$ Also every $S_n$-irreducible 
except the sign,  appears in the twisted conjugacy action on the set of odd permutations, for $n\geq 2.$ 

Equivalently:
For every partition $\lambda\vdash n,$ the Schur function $s_\lambda$ appears 
\begin{enumerate} 
\item in the (Schur-positive) symmetric function $ch\, \varepsilon(S_n, A_n)=
\sum_{\substack{\mu\vdash n\\ n-\ell(\mu) even}} E_\mu[F]$  for all $n\geq 5;$
\item in the (Schur-positive) symmetric function 
$ch\,\varepsilon(S_n,   \bar A_n) =\sum_{\substack{\mu\vdash n\\ n-\ell(\mu) odd}} E_\mu[F],$ except for $\lambda=(1^n),$ which never appears when $n\geq 2.$
\end{enumerate}
\end{thm}

\begin{proof}  Note that $H_\mu[F]$ and $E_\mu[F]$ coincide when $\mu$ is a partition with distinct parts.  Thus, in the proof of Theorem 4.17, we need to examine, and possibly modify the argument in  the places where the discussion involves partitions with repeated parts. 

 First, as before,  in Table 4 we record the decomposition into irreducibles for  $n\leq 12.$   Writing simply $\lambda$ for the $S_n$-irreducible indexed by the partition $\lambda,$ Maple  computation with SF shows that 
although the irreducible indexed by $(2^2)$ does {\it not} appear in 
$\varepsilon(S_4, A_4), $  for $(12\geq )\  n\geq 5$ 
$\varepsilon_n(S_n, A_n)$ contains all irreducibles, while  $\varepsilon_n(S_n, \bar A_n)$ contains all irreducibles except the sign for $(12\geq  )\ n\geq 4.$  

Our induction hypothesis will  be that for fixed $n\geq 12,$ whenever $5\leq k< n,$
(so in particular  {\bf k}$\neq $ {\bf 2,3,4}), for every partition $\lambda\neq (1^k)$ of $k,$ there are partitions $\tau(\lambda)$ and $\sigma(\lambda)$ of $k$ such that 
$\ell( \tau(\lambda))$ is even, $\ell( \sigma(\lambda))$ is odd, and 
$s_\lambda$ appears in both $E_{\tau(\lambda)}[F] $ and  $E_{\sigma(\lambda)}[F].$ When $\lambda= (1^n),$ we will show  that  it appears only in $E_{\tau(\lambda)}[F] $ for some $\tau$ such that $n-\ell( \tau(\lambda))$ is even.

\begin{enumerate}
\item Let $\lambda=(n).$ Then $\tau=(n)$ and $\sigma=(n-1,1)$ are partitions with lengths of opposite parity such that $s_{(n)}$ appears in $E_\tau[F]$ and $E_\sigma[F],$ for $n\geq 3. $  For $n=2,$ we have $E_{(2)}[F]=f_2=h_2.$  Hence $(n)$ appears in $\varepsilon(S_n,   \bar A_n)$ for all $n\geq 2,$ and in $\varepsilon(S_n,  A_n)$ if $n\geq 3.$
\item Let $\lambda=(1^n).$   Clearly the sign representation $s_{(1^n)}$ appears in $E_{(1^n)}[F]=e_n[f_1]=e_n,$ and $n-\ell((1^n))$ is even, so it appears in 
$\varepsilon(S_n,A_n).$  More generally, the sign representation cannot appear in $E_\lambda[F]$ if $\lambda$ has an even part, since 
for $f_n$ contains the sign only if $n$ is odd.  But if all parts of $\lambda$ are odd, then $n-\ell(\lambda)$ is necessarily even.  Hence 
the sign appears only in $\varepsilon(S_n,A_n)$ and never in 
$\varepsilon(S_n,   \bar A_n)$ ($n\geq 2$).
\item Let $\lambda=(n-1,1).$ Again for $n\geq 4,$ we may take 
$\tau(\lambda)=(n-1,1)$ and $\sigma(\lambda)=(n-2,1^2).$ 
In the first case $E_{(n-1,1)}[F]=f_{n-1} f_1$ as before, in the second case $E_{(n-2,1^2)}[F]=f_{n-2} e_2[f_1]=f_{n-2} e_2,$ which contains $h_{(n-2)}e_2 $ and hence contains $s_\lambda.$ If $n=3,$ 
 $s_{(2,1)}$ appears in $E_\sigma[F]$ only for  $\sigma=(2,1).$ Then $n-\ell(\sigma)=1$ is odd, and $E_{(2,1)}[F]$ contains $f_2 f_1=s_{(3)} +s_{(2,1)}.$  Hence ${(2,1)}$ appears only  in $\varepsilon(S_3, S_3\backslash A_3),$ but for $n\geq 4,$ $(n-1,1)$ appears in both representations.
\item Let $\lambda=(2,1^{n-2}).$ Since $f_r$ contains the sign representation if and only if $r$ is odd, it follows that 
\begin{enumerate}
\item If $n$ is odd, we may take $\tau(\lambda)=(n-2,2), n-2>2$ (because  then $E_\tau[F]=f_{n-2} h_2$ which contains $e_{n-2}h_2$) and $\sigma(\tau)=(n-2,1^2), n\geq 4$ (because $E_\sigma[F]=f_{n-2} e_2$ which now contains $e_{n-2} e_2,$ which also contains $s_{(2,1^{n-2})}$).  (The case $n=3$ has been addressed above.);
\item If $n$ is even, we may take $\tau(\lambda)=(n-1,1)$ and $\sigma(\lambda)=(n)$ for $n\geq 2; $  because the sign appears in $f_{n-1}$ and, by Lemma 4.7, $s_{(2,1^{n-2})}$ appears in $E_{(n)}[F]=f_n$ whenever $n$ is even, in which case $n-\ell((n))$ is odd.
\end{enumerate}
Hence ${(2, 1^{n-2})}$ appears in  $\psi(S_n, A_n)$ if $n\geq 4,$ and in $\psi(S_n,   \bar A_n)$ if $n\geq 3.$
\end{enumerate}

 We claim that the rest of the proof of Theorem 4.17 applies almost verbatim, with all $H$'s replaced by $E$'s. Our first  observation is that $H_\mu[F]=E_\mu[F]$ when $\mu$ has distinct parts.    Consequently, it is the two subcases of {\bf Case  4} that require additional scrutiny.  In both these cases, the analysis involves partitions in which the part equal to 1 occurs with multiplicity 2.  For {\bf Case  4-1}, the claim that needs to be verified is that $s_{\mu'}=s_{(q-1,1)}$ appears in $E_{q-2,1^2}[F],$ which is now equal to $f_{q-2} e_2[h_1]= f_{q-2} e_2.$ But this is clear since $f_{q-2}$ contains $h_{q-2}.$  When $m$ is odd, we need to check that $s_{(2,1^{m-2})}$ appears in $E_{(m-2,1^2)}[F]=f_{m-2} e_2,$ but this has already been verified in (4(a)) above. Similarly, in {\bf Case  4-2}, we need to verify the claim that $s_{\mu'}=s_{(2, 1^{q-2})}$ appears in $E_{(q-2, 1^2)}[F]=f_{q-2} e_2.$  Once again, since $q$ is odd, this has already been established in (4(a)) above.

The remaining parts of the proof of Theorem 4.17 pose no difficulty in the passage from $H[F]$ to $E[F],$  and our proof is complete.  \end{proof}

\begin{remark}
By similar arguments and careful analysis of products of Schur functions, we can  show that there is a single conjugacy class with the property that the conjugacy action of $S_n$ on this class contains every $S_n$-irreducible.  This is also true for the twisted conjugacy action.  The proof will appear elsewhere.
\end{remark}

From the Frobenius characteristic one can extract, by routine plethystic computations as in \cite{Su1}, more precise  information about the multiplicities of other irreducibles in $\psi(S_n).$  We state without proof:
\begin{prop}  The multiplicity in $\psi(S_n) $ 
\begin{enumerate} 
\item of the trivial representation is $p(n),$ the number of partitions of $n;$
\item of the sign representation is $|Par^{\neq}_n\cap Par^{odd}_n|,$ the number of partitions of $n$ all of whose parts are distinct and odd.
\item of the irreducible indexed by $(n-1,1)$ is $\sum_{\lambda\vdash n} (|\{i: m_i(\lambda)\geq 1\}|-1).$
\item of the irreducible indexed by $(2,1^{n-2})$ is 

$\sum_{\lambda\in Par^{\neq}_n\cap Par^{odd}_n} (\ell(\lambda)-1)$

+
$\sum (|\{i: m_i(\lambda)\geq 1\}|) $, 
where the second sum runs over the set of partitions ${\lambda\in Par_n^{odd}:m_i(\lambda)\leq 2, m_j(\lambda)=2 \text{ for exactly one part } j}.$

\end{enumerate}
\end{prop}

\begin{prop} The trivial representation and the sign representation each occur in $\varepsilon(S_n)$ as many times as there are partitions of $n$ with all parts odd, which is the same as the number of partitions of $n$ with distinct parts.
The irreducible indexed by $(n-1,1)$ (as well as the irreducible indexed by $(2,1^{n-2})$) occurs in 
$\varepsilon(S_n)$ with multiplicity
\begin{center}
$\sum_{\lambda\in Par^{\neq}_n} (\ell(\lambda)-1) + 
|\{\mu\vdash n: m_i(\mu)\leq 2\ \forall i, m_j(\mu)=2 \text{ for exactly one } j\}|.$\end{center}
\end{prop}

We end this section by considering one more class of representations.
For any fixed $k\geq 2,$  let $W_{n,k}=\sum_{\lambda\vdash n, \lambda_i=1 \text{ or } k} p_\lambda.$
We have
\begin{thm} $W_{n,k}$ is Schur-positive. 
\end{thm}

\begin{proof} Clearly $W_{n,k}=\sum_{r=0}^{\lfloor\frac{n}{k}\rfloor} p_k^r p_1^{n-kr}.$ Two observations are immediate:
\begin{itemize}
\item $W_{mk,k}$ is certainly an integer linear combination of Schur functions, and 
\item
if $n=mk+t$ with $0\leq t\leq k-1,$ then $W_{n,k}=p_1^t W_{mk,k}.$  
\end{itemize}

Thus it suffices to consider the case $n=mk.$ 
Define $\alpha_k=p_1^k-p_k, \beta_k=p_1^k+p_k.$ Using the fact that 
$p_k$ is an alternating sum of Schur functions indexed by hooks, i.e.,  $p_k=\sum_{r=0}^{k-1} (-1)^r s_{(k-r, 1^r)},$ we see that $\alpha_k$ and $\beta_k$ are both Schur-positive, since $f^\nu\geq 1$ for any partition $\nu.$

We have $p_1^k=\frac{\beta_k+\alpha_k}{2},$ 
$p_k=\frac{\beta_k-\alpha_k}{2}.$  
We compute

\begin{align*} W_{mk,k}&=\sum_{i=0}^{m} p_k^i p_1^{km-ki}
=\dfrac{(p_1^k)^{m+1}- p_k^{m+1}}{p_1^k-p_k} \\
&=\frac{1}{\alpha_k}\left(\left(\dfrac{\beta_k+\alpha_k}{2}\right)^{m+1}-\left(\dfrac{\beta_k-\alpha_k}{2}\right)^{m+1}\right)\\
&=\frac{1}{2^{m+1}\alpha_k} \left( \sum_{i=0}^{m+1} {m+1 \choose  i}\beta_k^{m+1-i} \alpha_k^i -\sum_{j=0}^{m+1} {m+1\choose   j}\beta_k^{m+1-j} \alpha_k^j (-1)^j\right) \\
&=\frac{1}{2^{m+1}\alpha_k}\sum_{\substack{j=1\\ j \text{ odd}}}^{m+1}{m+1\choose j} 2 \beta_k^{m+1-j} \alpha_k^j  = \frac{1}{2^m}\sum_{\substack{j=1\\ j \text{ odd}}}^{m+1}{m+1\choose j}  \beta_k^{m+1-j} \alpha_k^{j-1} \\
\end{align*}
It follows that $W_{mk,k}$ is a linear combination of Schur functions with coefficients that are nonnegative and rational.  But we already know that the coefficients must be integers, and hence $W_{mk,k}$ must be Schur-positive. \end{proof}

The special case $k=2$ is worth mentioning.  Here we have $W_{2n+1,2}=p_1 W_{2n,2}$ and 
$W_{2n,2}= \sum_{j \text{ odd}}{n+1\choose j}  h_2^{n+1-j} e_2^{j-1},$ because $\alpha_2=2 e_2$ and $\beta_2=2 h_2.$

\begin{qn}  Is there a unifying combinatorial interpretation of the following restricted row sums in the character table of $S_n,$ each of which is a nonnegative integer, for each $\nu\vdash n$? This is a subset of the list in Theorem 1.1; the first four below are {\it positive} integers for the values of $n$ and $\nu$ specified.  At present no combinatorial interpretation is known.
\begin{enumerate}
\item (Theorem 4.5) $c_\nu=\sum_{\lambda\vdash n} \chi^\nu(\lambda)$ for $n\neq 2.$
\item (Theorem 4.9) $c_\nu=\sum_{\lambda\in Par_n^{odd}} \chi^\nu(\lambda)$ 
\item (Corollary 4.12) $c_\nu=\sum_{\substack{\lambda\vdash n \\ n-\ell(\lambda) even}} \chi^\nu(\lambda)$ for $n\neq 2.$
\item (Corollary 4.18) $c_\nu=\sum_{\substack{\lambda\notin DO_n \\ n-\ell(\lambda) even}} \chi^\nu(\lambda)$
\item (Equation (4.11.3), Theorem 4.11) $c_\nu=\sum_{\lambda\notin DO_n } \chi^\nu(\lambda), $ $\nu\neq (1^n)$ when $n\geq 2.$
\end{enumerate}
\end{qn}

\section{The Lie and Foulkes characters}

In this section we will consider the more general characters obtained by inducing any $n$th root of unity from a cyclic subgroup $C_n$ of order $n$ up to $S_n.$ For fixed $k\geq 1,$ Foulkes showed \cite{Fo}  that the Frobenius characteristic of the induced representation 
$(e^{\frac{2i\pi}{n}})^k\uparrow_{C_n}^{S_n}$ is given by 
$f_n=\frac{1}{n}\sum_{d|n}\psi_k(d) p_d^{n/d}$ as in Section 3, with
\begin{equation}\psi_k(d)= \phi(d) \dfrac{\mu(\frac{d}{(d,k)})} {\phi(\frac{d}{(d,k)})}, \end{equation} where $(d,k)$ denotes the greatest common divisor of $d$ and $k$ as usual. The case  $k=n$ (or $k=0$) corresponds to the conjugation action of Section 4.  

Now let $k=1.$ Then  $\psi_1(d)=\mu(d),$ where $\mu$ is the number-theoretic M\"obius function, and  (see \cite{R}, \cite{St1}) $f_n= ch((e^{\frac{2i\pi}{n}})\uparrow_{C_n}^{S_n})$ is the Frobenius characteristic of the representation $Lie_n$ of $S_n$ on the multilinear component of the free Lie algebra on $n$ generators, while $f_n(t)$ is the dimension of the $n$th graded piece of the free Lie algebra on $t$ generators.  Also 
(\cite{St1}) $\pi_n=\omega Lie_n$ is the $S_n$-representation on the unique nonvanishing homology of the partition lattice $\Pi_n$.  Let 
 $\pi ^{alt}=\sum_{i\geq 1} (-1)^{i-1} \pi_i.$  Part (2) below was observed by direct calculation in \cite{HR}.

\begin{lem}
We have
\begin{enumerate}
\item 
$f_n(1)=0$ unless $n=1, $ in which case it is 1.  
\item  $f_n(-1) = (-1)\delta_{n,1}+ (1) \delta_{n,2};$ i.e., 
$f_1(-1)=-1, \ f_2(-1)=1,$ and $f_n(-1)$ equals zero otherwise.  
\end{enumerate}
\begin{proof} Part (1) is a standard property of the M\"obius function.  Part (2) follows from Lemma 3.3 applied to $\psi(d)=\mu(d),$  the M\"obius function, since if $n$ is even, 
$f_{2n}(1)=0,$ and $f_{n/2}(1)$ is nonzero if and only if $n=2,$ in which case it is 1.  \end{proof}
\end{lem}

By applying Theorem 3.2 with $v=1$ to this class of functions $f_n,$ we immediately obtain the  generating functions below.  The first  generating function in Corollary 5.2 is classical; it is essentially the statement of the Poincar\'e-Birkhoff-Witt theorem, see \cite{R} for an exhaustive literature. Part (2) is Cadogan's formula \cite{C},  \cite{Ha1}, and \cite[Solution to Problem 7.88]{St4EC2} for the calculation in this context) ). It also has a topological connection, arising as a consequence of the Whitney homology of the partition lattice  (\cite{Su1}). Part (3) appears in a recent paper  \cite{HR},where the authors derive it by direct calculation for the special case of the Lie character.  Part (4) below is a new observation.

\begin{cor} Writing $L$ for $F=\sum_{n\geq 1} f_n $ when  $f_n=ch(Lie_n)$ as above, we have the following plethystic generating functions:
\begin{enumerate}
\item   $  H[L](t) = \dfrac{1}{(1-t h_1)} $ 
\item  $ H[\pi^{alt}](t) = 1+t h_1= \pi^{alt}[H] (t).$ 
\item  $ E[ L](t) = \prod_{j\geq 1} (1-t^jp_j)^{f_j(-1)} =\dfrac{1-t^2 p_2}{1-t p_1} $ 
\item $ E[\pi^{alt}] (t)=\prod_{j\geq 1} (1+t^j p_j)^{-f_j(-1)}=\dfrac{1+t p_1}{1+t^2 p_2} $ 
%
\end{enumerate}
\end{cor}
\begin{proof} Combine Theorem 3.2, (1) to (4), and the  calculations recorded  in Lemma 6.1. 
\end{proof}

Contrast the next observation with Theorem 4.9 of the preceding section regarding the exterior powers of the conjugation action, in which every irreducible appears.

\begin{cor}  The degree $n$ term in $H[L]$ is the regular representation, and thus contains every irreducible of $S_n.$ The degree $n$ term in $E[L]$ is $2 p_1^{n-2} e_2,$ ($n\geq 2$), a representation of degree $n!$ which contains every irreducible except the trivial representation.  
\end{cor}
\begin{proof} A simple calculation shows that    the degree $n$ term of $E[L]$ is $2 p_1^{n-2} e_2,$ and hence the result. \end{proof}

We move on to the case of general $k$ and the function $\psi_k(d)$ given by equation (5.0.1).  Following Hardy and Wright (\cite{HW}, pp. 55-56) we define Ramanujan's sum
$$c_d(k) =\sum_{\substack{1\leq h\leq d\\ (h,d)=1}} \exp({2i\pi hk/d})= 
\sum_{\rho \text{ a primitive }d\text{th root of unity}}\rho^k.$$
Then we have, with the notation of equation (5.0.1), the following remarkable  formula of H\"older:
\begin{thm}  (\cite{H}, see also \cite[Theorems 271 and 272]{HW} for proofs) 
$$c_d(k)
             = \phi(d) \dfrac{\mu(\frac{d}{(d,k)})} {\phi(\frac{d}{(d,k)})}=\psi_k(d).$$
\end{thm}

Hence,  when $f_n$ is the Frobenius characteristic of the general Foulkes character $(e^{\frac{2i\pi}{n}})^k\uparrow_{C_n}^{S_n},$ 
we obtain the following result.  (Note the agreement with Lemmas 4.1 and 5.1 for the cases $k=n$ and $k=1$ respectively. )   
Part (1) below is originally due to von Sterneck \cite{vS}.  

\begin{lem} For fixed $k\geq 1,$ 
\begin{enumerate}
\item (\cite{vS}) 
$f_n(1)$ 
=$\begin{cases} 1, & \text{if }n|k;\\
                        0, &otherwise.\\
\end{cases}$
\item $f_n(-1)$ 
=$\begin{cases} -1, & \text{if } n \text{ is odd and } n|k;\\
                          1, & \text{if } n  \text{ is even and } \frac{n}{2}|k \text{ but }n\not{|} k;\\
                        0, &otherwise.\\
\end{cases}$
\end{enumerate}
\end{lem}
\begin{proof} We have 

 \begin{align*}
f_n(1)&=\frac{1}{n}\sum_{d|n} c_d(k)
=\frac{1}{n}\sum_{d|n} 
\sum_{\substack{\rho \text{ a primitive }\\d\text{th root of unity}} } \rho^k=\frac{1}{n}\sum_\tau \tau^k\\
%
\end{align*}
\noindent where the last sum ranges over all the $n$th roots of unity, since every $n$th root of unity is a primitive $d$th root of unity for a unique divisor $d$ of $n.$  
The result now follows because the $n$th roots of unity are the zeros of $(1-x^n),$ and if $k$ is not a multiple of $n,$ then the $k$th powers of the $n$th roots are simply a rearrangement of the  $(n,k)$th roots, each occurring $\frac{n}{(n,k)}$ times.


 Applying Lemma 3.3  to \begin{center}$f_n(-1) =\frac{1}{n}\sum_{d|n} c_d(k) (-1)^{\frac{n}{d}},$\end{center}
we immediately obtain the second statement. \end{proof}

\begin{thm} Fix $k\geq 1.$ Let $f_n$ be the Frobenius characteristic of the representation $\exp({2i\pi k/n})\uparrow_{C_n}^{S_n},$ and let 
$F=\sum_{n\geq 1} f_n.$ Then 
\begin{enumerate}
\item (Symmetric powers) 
\begin{equation}H[F](t)=\prod_{n\geq 1, n|k} (1-t^n p_n)^{-1}\end{equation} and hence 
\begin{equation}\sum_{\mu\vdash n}H_\mu[F]=\sum_{\lambda\in Par_n, \lambda_i|k} p_\lambda.\end{equation}
\item (Exterior powers) \begin{equation}E[F](t)=\prod_{n\geq 1, (2n-1)|k} (1-t^{2n-1} p_{2n-1})^{-1} 
\prod_{n\geq 1, n|k, (2n)\not|k}(1-t^{2n} p_{2n}),\end{equation} and hence 
\begin{equation}\omega(E[F])(t)=\prod_{n\geq 1, (2n-1)|k} (1-t^{2n-1} p_{2n-1})^{-1} 
\prod_{n\geq 1, n|k, (2n)\not|k}(1+t^{2n} p_{2n})\end{equation} and
\begin{equation}\sum_{\mu\vdash n}\omega(E_\mu[F])=\sum p_\lambda,\end{equation}
where the last sum ranges over all partitions $\lambda$ of $n$ such that the odd parts are factors of $k,$ the even parts occur at most once, and 
do {\it not} divide $k$, but half of each even part is a factor of $k.$
\item (Alternating exterior powers) 
$$\sum_{\lambda\in Par} t^{\lambda|} (-1)^{|\lambda|-\ell(\lambda)} \omega(E_\lambda[F]) $$
\begin{equation}
=\sum_{\lambda \in Par} t^{\lambda|}H_\lambda[\omega(F)^{alt}]
=H[\omega(F)^{alt}](t)
= \prod_{n\geq 1, n|k} (1+t^n p_n).\end{equation}

\item (Alternating symmetric powers) $$\sum_{\lambda\in Par} t^{\lambda|} (-1)^{|\lambda|-\ell(\lambda)} \omega(H_\lambda[F])
=\sum_{\lambda \in Par} t^{\lambda|}E_\lambda[\omega(F)^{alt}]$$
 \begin{equation}=E[\omega(F)^{alt}](t)
=\prod_{n\geq 1, (2n-1)|k} (1+t^{2n-1} p_{2n-1})
\prod_{n\geq 1, n|k, (2n)\not|k}(1+t^{2n} p_{2n})^{-1},\end{equation}
and hence \begin{equation}
\sum_{\lambda\in Par} t^{\lambda|} (-1)^{|\lambda|-\ell(\lambda)} H_\lambda[F]=\prod_{n\geq 1, (2n-1)|k} (1+t^{2n-1} p_{2n-1})
\prod_{n\geq 1, n|k, (2n)\not|k}(1-t^{2n} p_{2n})^{-1},\end{equation} a multiplicity-free nonnegative integer combination of power-sums.

\end{enumerate}
\end{thm}

\begin{proof} All parts follow from the corresponding parts of Theorem 3.2 and Lemma 5.5.  Parts (1) and (3) use the calculation of $f_n(1)$ in Lemma 5.5; Parts (2) and (4) use the value of $f_n(-1).$  Equation (5.6.8) follows by applying the involution $\omega$ to Equation (5.6.7).  Note that the right-hand side of (5.6.8) is now a nonnegative linear combination of power-sums.
\end{proof}

By combining (5.6.1) and (5.6.8) as in Theorem 3.4, we obtain more linear combinations of power-sums that are Schur-positive.  The same statement applies to adding and subtracting (5.6.5) and (5.6.6).  

Note also that when $k$ is prime, equation (5.6.1) coincides with the symmetric function $W_{n,k}$ of Theorem 4.23,  thereby  providing a concrete realisation of the representation. In particular when $k=2$ we obtain, for the symmetric function $W_{n,2}$ which is a homogeneous polynomial of degree $n$ in the power-sums $p_1$ and $p_2:$

\begin{cor} Let $f_n$ be the Frobenius characteristic of the representation $\exp{(4\pi i /n)}\uparrow_{C_n}^{S_n}.$ Then 
\begin{center}
$W_{n,2}=\sum_{\mu\vdash n}H_\mu[F], $\\
$(1+p_1)(1-p_4)^{-1}\vert_{\text{deg } n} =g_n=
\sum_{\mu\vdash n}(-1)^{n-\ell(\mu)}H_\mu [F], $\end{center}

and the following polynomials in the  power-sums $p_1, p_2$ and $p_4$ are Schur-positive:
\begin{eqnarray}
& \frac{1}{2} (W_{n,2}+g_n)= 
 \sum_{\substack {\mu\vdash n\\ n-\ell(\mu)\text{ even}}}H_\mu [F];\\
& \frac{1}{2} \left(W_{n,2}-g_n\right)= 
 \sum_{\substack {\mu\vdash n\\ n-\ell(\mu)\text{ odd}}}H_\mu [F].
\end{eqnarray}
\end{cor}

A similar statement can be made for the case of odd prime $k.$ 
\section{The alternating group}
In this section we examine yet another representation of $S_n$ with the same degree as the regular representation. 
Consider the alternating group $A_n$ operating on itself by conjugation, a representation of degree $n!/2$ which we denote by $\psi(A_n).$ When induced up to $S_n$ it gives a representation $\psi(A_n)\uparrow_{A_n}^{S_n}$  of degree $n!.$     Before we prove our next result, we investigate this situation in greater  generality.  

\begin{prop} Let $G$ be a finite group and $H$ a subgroup of $G$ of index 2 (hence normal).  Fix $g_0\notin H,$  so that  $G=H\cup g_0H$ is  a disjoint union of left cosets.  Let $\psi(H)$ denote the action of $H$ on itself by conjugation, and consider the $G$-dimensional induced $G$-module  $\psi(H)\uparrow_H^G.$ Let $\chi_C(H)$ denote the character of $\psi(H),$ and let $Z(\sigma)$ be the centraliser of $\sigma$ in $G.$ Then we have

\begin{enumerate} 
\item For $\sigma\in H,$ $\chi_C(H)(\sigma)=|H\cap Z(\sigma)|$
\item For $\tau\in G,$ $\chi_C(H)\uparrow_H^G(\tau)$

$=\begin{cases} 0, &\tau\notin H;\\
                       2\chi_C(H)(\tau) &otherwise.
\end{cases}$
\end{enumerate} 

\end{prop}

\begin{proof}  Part (1) is a consequence of the easy observation that the  trace of $\sigma$ acting by conjugation on H is the number of $h\in H$ such that $\sigma h= h\sigma.$ 

For Part (2), we use 
the standard formula for the character of an induced representation.  This gives 
$$\chi_C(H)\uparrow_H^G(\tau)= \chi_C(H) (\tau) \delta_{\tau\in H}
+ \chi_C(H) (g_0\tau g_0^{-1}) \delta_{g_0\tau g_0^{-1}\in H},$$
where we are writing $\delta_P = 1$ if statement $P$ is true, and 0 otherwise.

Since $H$ is normal in $G,$ $\tau\in H\iff g_0\tau g_0^{-1}\in H.$ 
Hence the induced character has value 

$\chi_C(H)\uparrow_H^G(\tau)$

$=\begin{cases} 0, &\text{ if } \tau\notin H\\
                         2 |H\cap Z(\tau)|, &\text{ if } \tau\in H 
\text{ and  } g_0\tau g_0^{-1} \text{ is conjugate to } \tau \text{ in } H,\\
                         |H\cap Z(\tau)|+|H\cap Z(g_0\tau g_0^{-1})|,
 &\text{ if } \tau\in H 
\text{ and  } g_0\tau g_0^{-1} \text{ is NOT conjugate to } \tau \text{ in } H
\end{cases}$

$=\begin{cases} 0, &\tau\notin H,\\
                       2\chi_C(H)(\tau), &\tau\in H \text{ and } g_0\tau g_0^{-1}= h\tau h^{-1} \text{ for some } h\in H;\\
                       \chi_C(H)(\tau) + \chi_C(H)(g_0\tau g_0^{-1}), &\text{ otherwise}
\end{cases}$     

Finally note that $Z(\sigma)\cap H =Z(g_0\sigma g_0^{-1}) \cap H,$         
because $x\in Z(\sigma)\iff  g_0 x g_0^{-1}\in Z( g_0\sigma g_0^{-1})$ and $x\in H \iff g_0 x g_0^{-1}\in H.$  This implies that $|H\cap Z(\tau)|+|H\cap Z(g_0\tau g_0^{-1})|=2|H\cap Z(\tau)|$  if $\tau\in H.$ The result follows.  \end{proof}

\begin{prop}  Let $G$ be a finite group and $H$ a subgroup of $G$ of index 2 (hence normal).  Fix $g_0\notin H,$  so that  $G=H\cup g_0H$ is  a disjoint union of left cosets.  Consider the action $\psi(G)$ of $G$ on itself by conjugation.  Then 
\begin{enumerate} 
\item $H$ and $g_0 H$ are disjoint orbits for this action.  
\item Hence if $\psi(G,H), \psi(G, g_0 H)$ denote respectively the actions of $G$ by conjugation on $H$ and $g_0 H,$ we have 
$$\psi(G) = \psi(G,H) \oplus \psi(G, g_0 H).$$ 
\item The character values of $\psi(G,H)$ and $\psi(G, g_0 H);$ on $\tau \in G$ are  equal to $|Z(\tau)\cap H|$ and 
$|Z(\tau)\cap g_0 H|$ respectively.
\item Let $\omega_H$ denote the one-dimensional representation of $G$  induced by the unique nontrivial  one-dimensional representation of the quotient group $G/H$ (which is cyclic of order 2), with character $\chi_{\omega_H}(g)$ 

 $ =  \begin{cases} 1, & g\in H;\\
                                                 -1, & g\notin H \iff g\in g_0 H.\\
                              \end{cases}$

Then $$\psi(H)\uparrow_H^G= \psi(G,H) \oplus \omega_H \otimes \psi(G,H),$$   and hence $\omega_H\otimes \psi(H)\uparrow_H^G =\psi(H)\uparrow_H^G.$
\item $\psi(H)\uparrow_H^G$ is a $G$-submodule of 
$   \psi(G) \oplus \omega_H \otimes \psi(G).$  In fact if we define 
$U_G^+$ to be $\left(\psi(G) \oplus \omega_H \otimes \psi(G)\right) / \psi(H)\uparrow_H^G,$ then we have 
$$U_G^+= \psi(G, g_0 H)\oplus \omega_H\otimes \psi(G, g_0 H),$$ and hence $\omega_H\otimes U_G^+ =U_G^+.$
\item Define $U_G^{*}$ to be the virtual module $\psi(G,H)-\psi(G, g_0 H).$ Then 
$$2 \psi(H)\uparrow_H^G=(\psi(G)+\omega_H \otimes\psi(G)) 
+ (U_G^* + \omega_H \otimes U_G^*).$$
\item $$\psi(H)\uparrow_H^G= U_G^+ + (U_G^* + \omega_H \otimes U_G^*).$$
\item $$U_G^+  +  U_G^*=\psi(G,H)+ \omega_H \otimes\psi(G, g_0 H).$$
\end{enumerate}

\end{prop}

\begin{proof}For (1)and (2): this is clear, since $H$ is normal in $G$ and $gxg^{-1}\in H
\iff x\in g^{-1}Hg=H,$ and hence $g(g_0 h) g^{-1}\in g_0 H$ for all $h\in H.$

 For (3):  Let $\tau\in G.$ We compute the trace of $\tau$ acting on H by conjugation:
$$\text{tr\ } (\tau|_H)= |\{h\in H: \tau h\tau^{-1}=h\}|= |\{h\in H: \tau h=h\tau\}| =|Z(\tau)\cap H| $$ 
and similarly for $\tau$ acting on $g_0H$ by conjugation: $\text{tr\ } \tau|_{g_0 H}
= |Z(\tau )\cap g_0 H|. $

For (4): From the trace of $\tau$ acting on $H$ it follows that 
the character value of $\tau$ on the representation $\psi(G,H) \oplus ( \omega_H \otimes \psi(G,H))$ equals 
$(1+\chi(\omega_H)(\tau))\cdot  |Z(\tau)\cap H|$ which in turn equals 

$\begin{cases} 0, & \tau\notin H, \text{ since then } \chi(\omega_H)(\tau)=-1\\
2 |Z(\tau)\cap H|, &else. \\
\end{cases}$  The second statement follows since $\omega_H\otimes \omega_H$ is the trivial representation of $G.$ 

For (5): This follows immediately from (2) and (4).  

The statement in (6) is immediate upon substituting from (2) and (4), since $\psi(G)+U_G^* = 2\psi(G,H),$ with the analogous result when we tensor with $\omega_H.$   Statement (7) follows by using the definition of $U_G^+,$ and (8) follows from (5). \end{proof}

%

Now we apply this to the alternating subgroup $A_n$ of $S_n.$  For $\sigma$ in $S_n,$ if $\sigma$ has cycle-type $\lambda$ and $\lambda$ has $m_i$ parts equal to $i,$ recall that $z_\lambda=\prod_{i\geq 1} i^{m_i} (m_i!)$ is the order of the centraliser of $\sigma.$ 

\begin{lem} For any $\sigma\in S_n$ with cycle-type given by a partition $\lambda $ of $n,$  $|Z(\sigma)\cap A_n|$

$=\begin{cases} 
|Z(\sigma)| = z_\lambda, &\sigma\in A_n \text{ and  all parts of } \lambda \text{ are odd and distinct };\\
                        \frac{1}{2} z_\lambda, \text{ otherwise}.
 \end{cases}$
\end{lem}

\begin{proof}
It is well-known that for any subgroup $K$ of $S_n,$ either $K\subset A_n$ or exactly half the permutations of $K$ are even.  Hence for any $\sigma\in S_n$ with cycle-type given by a partition $\lambda $ of $n,$  $|Z(\sigma)\cap A_n|$

$=\begin{cases}|Z(\sigma)| = z_\lambda, &\text{ if }Z(\sigma)\subset A_n\\
                        \frac{1}{2} z_\lambda, &\text{ otherwise, i.e., if } 
Z(\sigma) \text{ contains an odd permutation} \end{cases}$

  Let $\sigma\in S_n$  be of cycle-type $\lambda\vdash n.$ We claim that  the centraliser $Z(\sigma)$ contains an odd permutation if and only if some part $i$ is even, or all parts $i$ are odd but some odd part is repeated. But  the natural embedding of $S_a[S_b]$ in $S_{ab}$ (e.g., \cite{JK}) shows that there are odd permutations in the wreath product subgroup $S_a[S_b]$ if and only if  $b$ is odd and $a\geq 2,$  or $b$ is even.   Since $Z(\sigma)$ is isomorphic to the direct product $\times_i S_{m_i}[S_i],$ 
the conclusion follows.
\end{proof}

Hence we have:
\begin{thm}  The induction $\psi(A_n)\uparrow_{A_n}^{S_n}$ is a self-conjugate representation of $S_n$ of degree $n!$ with Frobenius characteristic 
\begin{center} $ 2\sum_{\lambda \in DO_n}
p_\lambda + \sum_{\substack{\mu \notin DO_n\\ n-\ell(\mu) even}}
p_\mu.$\end{center}  
We also have 
\begin{center} $\psi(A_n)\uparrow_{A_n}^{S_n}=\psi(S_n, A_n) +\omega(\psi(S_n, A_n)).$\end{center}
For $n\neq 3,$ $\psi(A_n)\uparrow_{A_n}^{S_n}$ has the property that every $S_n$-irreducible appears in its irreducible decomposition, 
while for  $n=3,$ $ch (\psi(A_3)\uparrow_{A_3}^{S_3})= 3 (h_3+e_3).$
\end{thm} 

\begin{proof}Apply the Frobenius characteristic map to the character $\chi$ of  $\psi(A_n)\uparrow_{A_n}^{S_n}:$ 

\begin{center}$ch(\chi)=\sum_{\lambda\vdash n} z_\lambda^{-1} p_\lambda \chi(\lambda).$\end{center}

Note that $\lambda$ is the cycle-type of an even permutation if and only if $\sum_i (i-1)m_i(\lambda)=n-\ell(\lambda)$ is even.

Hence from Proposition 6.2 and Lemma 6.3, $\chi(\lambda)$

$=\begin{cases} 0 &\text{if } n-\ell(\lambda) \text{ is odd;}\\
                     2z_\lambda &\text{if  all parts of } \lambda \text{ are                               
                                                      odd and distinct };\\ 
                      z_\lambda &\text{otherwise, with } n-\ell(\lambda) \text{  even},
\end{cases}$

\noindent and the result follows.  The penultimate statement is a direct consequence of Proposition 6.2 (4), and hence, by Theorem 4.17, the last statement follows as well. \end{proof}

\begin{prop}  We have
\begin{align}
\psi(A_n)\uparrow_{A_n}^{S_n}&=\psi(S_n,A_n)+\omega(\psi(S_n, A_n))\\ &=\sum_{\substack{\mu\vdash n\\ n-\ell(\mu) \, even}}H_\mu[F] 
+ \sum_{\substack{\mu\vdash n\\ n-\ell(\mu) \, even}}\omega(H_\mu[F] )\\
 \psi(A_n)\uparrow_{A_n}^{S_n} &= U_n^+ + 2 U_n^{DO}\\
 2 \psi(A_n)\uparrow_{A_n}^{S_n}
&=\psi(S_n) + \omega(\psi(S_n)) +2 U_n^{DO}\\ 
U_n^{+} &= (\psi(S_n) + \omega(\psi(S_n))) -
 \psi(A_n)\uparrow_{A_n}^{S_n}
\end{align}
\end{prop}

\begin{proof}  We use Proposition 4.13.
Equations (6.5.1) and (6.5.2) follow from Theorem 6.4. 
Equation (6.5.3) is just a matter of putting together equations (4.13.5) and (4.13.2) of Proposition 4.13, and equation (6.5.1) above. Equation 
(6.5.4) is simply (6) of Proposition 6.2, coupled with the fact that $U_n^{DO}$ is self-conjugate (equation (4.13.1)).

 Finally, equation (6.5.5)  follows by eliminating $\psi(A_n)\uparrow_{A_n}^{S_n}$ from  equations (6.5.4) and (6.5.3), or directly from Proposition 6.2 (5)-(7).   \end{proof}

The irreducible representations of $A_n$ are determined essentially from Clifford theory for the special case of a finite group $G$ with a subgroup $H$ of index 2.  The results are summarised in the following theorem.  The reference \cite{Sim} derives all the basic facts without using Clifford theory; in contrast \cite{CSST} is an exposition using Clifford theory.  (These two references were supplied by an anonymous referee, whom the author gratefully acknowledges.)

\begin{thm} (see \cite[Proposition 5.1]{FH}, \cite{Sim}, \cite{CSST}) Let $G$ be a finite group  with a subgroup $H$ of index 2.  Let $\omega_H$ be the linear character induced by $G/H$ as in Proposition 6.2. Let $Irr(G)$ (respectively $Irr(H)$) denote the irreducible characters of $G$ (respectively $H$). Let $\chi\in Irr(G).$ Then 
\begin{enumerate}
\item If  $\chi(g)\neq 0$ for some $g\notin H,$ then $\chi$ and $\omega_H\cdot \chi\in Irr(G),$ and ${\chi\downarrow}_H =\omega_H\cdot {\chi\downarrow}_H $ is also an irreducible for $H.$ 
\item If $\chi(g)=0$ for all $g\notin H,$ then $\chi\downarrow_H$ is a sum of two irreducible characters $\chi^+$ and $\chi^-$ in $H$. 
\item All irreducibles of $H$ arise in this manner.  
\item If $\psi\in Irr(H),$ there exists $\chi\in Irr(G) $ such that $\psi$ is an irreducible component of $\chi\downarrow_H$, or equivalently, 
such that $\chi$ is an irreducible component of $\psi\uparrow^G.$ 
In this case either $\chi\downarrow_H$ is irreducible so it equals $\psi,$ or $\chi\downarrow_H$ splits into two irreducibles, and then $\psi$ is either $\chi^+$ or $\chi^-.$
\item Let $W$ be any representation of $H$.  Suppose $W$ decomposes into $H$-irreducibles as 
$$W=\sum_{\substack{\chi\in Irr(G)\\ \chi(g)\neq 0 \text{ for some } g\notin H}} c_\chi \chi
+\sum_{\substack{\chi\in Irr(G)\\ \chi(g)= 0 \text{ for all } g\notin H}} d_+( \chi^+) + d_{-}(\chi^-),$$  where in the first sum we take exactly one of $\chi,$ $\omega_H\cdot \chi.$ Then 
$W\uparrow_H^G$ decomposes into $G$ irreducibles as 
$$\sum_{\substack{\chi\in Irr(G)\\ \chi(g)\neq 0 \text{ for some } g\notin H}} c_\chi (\chi +\omega_H\cdot \chi)
+\sum_{\substack{\chi\in Irr(G)\\ \chi(g)= 0 \text{ for all } g\notin H}} (d_+ + d_{-})\chi,$$  where in the first sum we take exactly one of $\chi,$ $\omega_H\cdot \chi.$
\end{enumerate}
\end{thm}

\begin{cor} Let $G$ be a finite group with a subgroup $H$ of index 2.  Let  $W$ be a representation of $H.$ Then $W$ contains every irreducible of $H$ if and only if  the induced representation $W\uparrow_H^G$ contains every irreducible of $G.$
\end{cor}

\begin{proof} Immediate from (5) of the preceding Theorem 6.6. \end{proof} 

In the special case of the  alternating group we immediately obtain, from Theorem 6.4 and Corollary 6.7,  the following known result:


\begin{cor} The conjugacy action of $A_n$ on itself contains every ($A_n$-)irreducible for $n\neq 3$.  
\end{cor}

The preceding corollary is also a consequence of  the following deep result of Heide, Saxl, Tiep and Zalesski:
\begin{thm} \cite[Corollary 1.4]{HSTZ} If $G$ is a  finite nonabelian simple group, then every $G$-irreducible appears in the conjugation action of $G$ on itself. 
\end{thm}

Hence we have 

\begin{thm} If $G$ is a finite group and $H$ is a subgroup of index 2 which is nonabelian and simple, then the conjugacy action of $H$ on itself, when induced up to $G,$ contains every $G$-irreducible.
\end{thm}

We conclude with some natural questions raised by this section, in analogy with the results established for the symmetric group.

\begin{qn}  Given a group $G$ with a subgroup $H$ of index 2, we know (Proposition 6.2)  that $\psi(G)$ decomposes into a direct sum 
$\psi(G,H)\oplus \psi(G, g_0 H).$ If $\psi(G)$ contains every $G$-irreducible, is this also true for the submodules $\psi(G,H)$, $\psi(G, g_0 H)$  (the analogues of $S_n$ acting by conjugation on the subsets of even and odd permutations), and the representations $U_G^+$ and $U_G^+ \oplus U_G^* $ defined in Proposition 6.2?  

\end{qn}

Tables 1 and 2 which follow contain the decomposition into irreducibles up to $n=10$ for the ordinary and twisted conjugacy representations $\psi(S_n)$ and $\varepsilon(S_n).$  In these tables, the rows index the irreducibles according to partitions of $n$ listed in reverse lexicographic order.   Tables 3 and 4 list the decomposition of the ordinary and twisted conjugacy actions respectively, on the subset of even permutations, and the set of odd permutations, up to $n=12.$



\vfill\eject
\begin{center} Table 1: Decomposition into irreducibles of conjugacy action of $S_n$

\begin{tabular}{l| r r r r r r r r  r r r r r r r r }
$\qquad  n$ &1 &2 & 3 & 4 & 5 & 6 & 7 & 8 & 9 & 10& 11 & 12 & 13 & 14 & 15 &16\\
$\lambda$ \small{\ reverse\ lex\ order} \\
\hline
$(n)$ &1  & 2 & 3 &5 &7 & 11 & 15 & 22 & 30 & 42 &56 & 77 & 101 & 135 & 176 & 231\\
 $(n-1,1)$ & & 0 & 1 & 2 & 5 & 8 & 15 & 23 & 37 & 55\\ 
$(n-2, 2)$ & &  & 1 & 3 & 6 & 15 & 26 & 49& 79 & 131\\
$(n-2, 1^2)$ & &  & & 2 & 5 & 10 & 19 &33 & 57 & 91\\
& &  & & 1 & 4 & 4 & 18 & 39 & 87 & 157\\
& &  & & &  3 & 13 & 36 & 78 & 154 & 284\\
& &  & & &  1 & 10 & 19 & 44 & 82 & 148\\
& &  & & &      & 8 & 18 & 25 & 64 & 165\\
& &  & & &      & 5  &22 & 70 & 188 & 424\\
& &  & & &      &  4 & 28 & 67 & 152 & 327\\
& &  & & &      & 1 & 7 & 81 & 201 & 427\\
& &  & & &      &    & 12 & 34 & 75 & 158\\
& &  & & &      &     & 10 & 35 & 95 & 52\\
& &  & & &      &     & 5 & 53 & 168 & 345\\
& &  & & &      &     &  1 & 58 & 207 & 497\\
& &  & & &      &     &     & 52 & 203 & 614\\
& &  & & &      &     &     & 17 & 169 & 546\\
& &  & & &      &     &     &  19 & 52 & 447\\
& &  & & &      &     &     &  19 & 41 & 124\\
& &  & & &      &     &     & 17 & 144 & 284\\
& &  & & &      &     &     & 5 & 104 & 283\\
& &  & & &      &     &     & 2 &   81 & 200\\
& &  & & &      &     &     &    &  130 & 721\\
& &  & & &      &     &     &    &    84 & 482\\
& &  & & &      &     &     &    &     23 & 305\\
& &  & & &      &     &     &    &     34 & 492\\
& &  & & &      &     &     &    &     39 & 311\\
& &  & & &      &     &     &    &      21 & 72\\
& &  & & &      &     &     &    &      7  & 194\\
& &  & & &      &     &     &    &     2  & 208\\
& &  & & &      &     &     &    &       & 387\\
& &  & & &      &     &     &    &       & 177\\
& &  & & &      &     &     &    &       & 241\\
& &  & & &      &     &     &    &       & 258\\
& &  & & &      &     &     &    &       & 128\\
& &  & & &      &     &     &    &       & 33\\
& &  & & &      &     &     &    &       & 46\\
& &  & & &      &     &     &    &       & 65\\
& &  & & &      &     &     &    &       & 63\\
& &  & & &      &     &     &    &       & 25\\
& &  & & &      &     &     &    &       & 9\\
& &  & & &      &     &     &    &       & 2\\

\end{tabular}
\end{center}
\vfill\eject

\begin{center} Table 2: Decomposition into irreducibles of twisted conjugacy action of $S_n$ 

\begin{tabular}{l| r r r r r r r r  r r r r r r r r }
$\qquad  n$ &1 &2 & 3 & 4 & 5 & 6 & 7 & 8 & 9 & 10& 11 & 12 & 13 & 14 & 15 &16\\
$\lambda$  \small{\ reverse\ lex\ order}\\
\hline
$(n)$ &1  & 1 & 2 &2 &3 & 4 & 5 & 6 & 8 & 10 \\
 $(n-1,1)$ & & 1 & 1 & 3 & 4 & 6 & 9 & 13 & 17 & 23\\ 
$(n-2, 2)$ & &  & 2 & 1 & 4 & 8 & 14 & 23& 36 & 53\\
$(n-2, 1^2)$ & &  & & 3 & 7 & 12 & 19 &29 & 44 & 63\\
                   & &  & & 2 & 4 & 6 & 14 & 29 & 55 & 92\\
                   & &  & &   &4 & 13 & 33 & 65 & 114 & 193\\
                   & &  & &    &3 & 12 & 23 & 42 & 72 & 115\\
                   & &  & &   &   &  6 & 19 & 13 & 39 & 93\\
                   & &  & & &      & 8  &19 & 66 & 160 & 329\\
                   & &  & & &       & 6 & 33 & 53 & 121 & 236\\
                   & &  & & &       & 4 & 19& 89 & 196 & 382\\
                   & &  & & &      &     & 14 & 42 & 83   &156\\
                   & &  & & &       &      &14 & 37 & 82 & 41\\
                   & &  & & &       &     &9  & 53 & 153 & 280\\
                   & &  & & &         &  &5 & 66  & 208 & 433\\
                   & &  & & &         &    &   & 65 & 208 & 566\\
                   & &  & & &          &    &    & 29 & 196 & 525\\
                   & &  & & &      &     &       &  13 & 72 & 473\\
                   & &  & & &      &     &       & 29 & 43 & 156\\
& &  & & &      &     &     & 23 & 153 & 237\\
& &  & & &      &     &     & 13 &   82 & 289\\
& &  & & &      &     &     & 6   &  121 & 196\\
& &  & & &      &     &     &    &    160 & 721\\
& &  & & &      &     &     &    &     114 & 525\\
& &  & & &      &     &     &    &     44 & 289\\
& &  & & &      &     &     &    &     39 & 566\\
& &  & & &      &     &     &    &      55 & 382\\
& &  & & &      &     &     &    &      36  & 115\\
& &  & & &      &     &     &    &     17  & 196\\
& &  & & &      &     &     &    &      8 & 237\\
& &  & & &      &     &     &    &       & 433\\
& &  & & &      &     &     &    &       & 236\\
& &  & & &      &     &     &    &       & 280\\
& &  & & &      &     &     &    &       & 329\\
& &  & & &      &     &     &    &       & 193\\
& &  & & &      &     &     &    &       & 63\\
& &  & & &      &     &     &    &       & 41\\
& &  & & &      &     &     &    &       & 93\\
& &  & & &      &     &     &    &       & 92\\
& &  & & &      &     &     &    &       & 53\\
& &  & & &      &     &     &    &       & 23\\
& &  & & &      &     &     &    &       & 10\\

\end{tabular}
\end{center}
\vfill\eject
\begin{center} Table 3: Irreducible decomposition of 
$\psi(S_n, A_n)$ and $\psi(S_n, \bar A_n),$  $2\leq n\leq 12:$
\end{center}
\begin{tiny} \begin{eqnarray*}  &\psi(S_2, A_2) &=   {(2)}= \psi(S_2, \bar A_2)\\
&\psi(S_3, A_3) &=2   {(3)}+  {(1^3)}; \quad \psi(S_3, \bar A_3)=  {(3)}+  {(2,1)}\\
&\psi(S_4, A_4) &= 3   {(4)}+  {(3,1)}+   {(2^2)}+  {(2,1^2)}+   {(1^4)},\\  
&\psi(S_4, \bar A_4) &=2   {(4)}+  {(3,1)} +2   {(2^2)}+   {(2,1^2)};\\
&\psi(S_5, A_5) &= 4   {(5)}+2  {(4,1)}+ 3  {(3,2)}+3  {(3,1^2)} + 2  {(2^2,1)} +  {(2,1^3)}+   {(1^5)},\\ 
&\psi(S_5, \bar A_5) &=3   {(5)}+3  {(4,1)} +3  {(3,2)}+ 2   {(3,1^2)} + 2  {(2^2,1)} +2  {(2,1^3)};\\
&\psi(S_6, A_6) &= 6    {(6)}+4  {(5,1)}+ 7  {(4,2)}+5  {(4,1^2)} + 2  {(3^2)}+ 7  {(3,2,1)} 
+5   {(3,1^3)} + 4  {(2^3)} + 2   {(2^2,1^2)} +2  {(2,1^4)} +   {(1^6)},\\ 
&\psi(S_6, \bar A_6) &=5   {(6)}+4  {(5,1)}+ 8  {(4,2)}+5  {(4,1^2)} + 2  {(3^2)}+ 6  {(3,2,1)} 
+5   {(3,1^3)} + 4  {(2^3)} + 3   {(2^2,1^2)} +2  {(2,1^4)}.\\ 
&\psi(S_7, A_7) &= 8    {(7)}+7  {(6,1)}+ 13  {(5,2)}+10  {(5,1^2)} + 9{(4,3)}+ 18  {(4,2,1)} 
+ 9  {(4,1^3)}\\
& & + 9  {(3^2,1)} + 11   {(3,2^2)}+ 14  {(3,2,1^2)} +4  {(3,1^4)}
  + 6   {(2^3,1)}+ 5  {(2^2,1^3)}
+2   {(2,1^5)}+    {(1^7)},\\ 
&\psi(S_7, \bar A_7) &= 7    {(7)}+8  {(6,1)}+ 13  {(5,2)}+9  {(5,1^2)} + 9  {(4,3)}+ 18  {(4,2,1)} 
+ 10  {(4,1^3)}\\
& &  + 9  {(3^2,1)} + 11   {(3,2^2)}+ 14  {(3,2,1^2)}+3  {(3,1^4)}
   + 6   {(2^3,1)}+ 5  {(2^2,1^3)}
+3   {(2,1^5)};\\
&\psi(S_8, A_8) &=    12 {(8)}+11  {(7,1)}+24   {(6,2)}+17  {(6,1^2)} + 20 {(5,3)}+39  {(5,2,1)} 
+ 22  {(5,1^3)}
 +12   {(4^2)} +35   {(4,3,1)}\\
&&+34   {(4,2^2)}
+40 (4,2,1^2)+ 17 {(4,1^4)} + 17  {(3^2,2)}+  27(3^2,1^2) + 29  {(3,2^2,1)}+26   {(3,2,1^3)}\\
&&+9    {(3,1^5)} 
+ 9 (2^4)  +10 (2^3, 1^2) +8 (2^2, 1^4)  +2  (2, 1^6)  +  2(1^8);\\
&\psi(S_8, \bar{A_8}) &=  10   {(8)}+12  {(7,1)}+  25 {(6,2)}+ 16 {(6,1^2)} + 19 {(5,3)}+ 39 {(5,2,1)} 
+  22 {(5,1^3)}+ 13  {(4^2)} +  35 {(4,3,1)}\\
&&+  33 {(4,2^2)} +41 (4,2,1^2)+  17{(4,1^4)}  +18   {(3^2,2)}
+26 (3^2,1^2)+ 29  {(3,2^2,1)}
+ 26  {(3,2,1^3)}\\
&&+  8  {(3,1^5)}
 + 10 (2^4)  +9 (2^3, 1^2) +9 (2^2, 1^4)  + 3 (2, 1^6)  ;\\
& \psi(S_9,A_9) &= 16 (9) +18 (8,1) + 39 (7,2) +29 (7,1^2) +44 (6,3) +77 (6,2,1)+ 41 (6,1^3)\\
& & + 32 (5,4) +94 (5,3,1) +76 (5,2^2) + 100 (5,2,1^2) +38 (5, 1^4)\\
& & +47 (4^2,1) + 84 (4,3,2) + 104 (4,3,1^2) +102 (4,2^2,1) + 84 (4,2,1^3) +26 (4, 1^5)\\
& & + 20 (3^3) +72 (3^2,2,1) + 52 (3^2, 1^3) +40 (3, 2^3) + 65 (3,2^2,1^2) +42 (3,2,1^4) \\
& & +12 (3,1^6)+17 (2^4,1) +20 (2^3, 1^3) + 10 (2^2, 1^5) +3 ( 2, 1^7) +2 (1^ 9)\\
& \psi(S_9, \bar{A_9}) &=14 (9) +19 (8,1) + 40 (7,2) +28 (7,1^2) +43 (6,3) +77 (6,2,1)+ 41 (6,1^4)\\
&&+ 32 (5,4) +94 (5,3,1) +76 (5,2^2) + 101 (5,2,1^2) +37 (5, 1^4)\\
& & +48 (4^2,1) + 84 (4,3,2) + 103 (4,3,1^2) +101 (4,2^2,1) + 85 (4,2,1^3) +26 (4, 1^5)\\
& & + 21 (3^3) +72 (3^2,2,1) + 52 (3^2, 1^3) +41 (3, 2^3) + 65 (3,2^2,1^2) +42 (3,2,1^4) \\
&&+11 (3,1^6)+17 (2^4,1) +19 (2^3, 1^3) + 11 (2^2, 1^5) +4 ( 2, 1^7) ;\\
& \psi(S_{10},A_{10}) &=22 (10)+27 (9,1) + 65 (8,2) +46 (8,1^2) +79 (7,3) + 142 (7,2,1) +74 (7, 1^3)\\
& &  + 82 (6,4) + 212 (6,3,1) + 164 (6,2^2) + 213 (6,2,1^2) +79 (6, 1^4) \\
& & +26 (5^2) + 173 (5,4,1) + 248 (5,3,2) +307 (5,3,1^2) + 273 (5,2^2,1) + 224 (5,2, 1^3) +62 (5, 1^5)\\
& & + 142 (4^2,2) + 141 (4^2, 1^2) +100 ( 4,3^2) + 361(4,3,2,1) +241 (4,3,1^3) + 152 (4,2^3)\\
&& \qquad\qquad\qquad\qquad\qquad\qquad\qquad\qquad\qquad + 246 ( 4,2^2,1^2) + 155 ( 4,2,1^4) + 36 (4, 1^6)\\
& & + 97 (3^3,1)  + 104 (3^2, 2^2)+ 193 ( 3^2,2,1^2) + 89 (3^2, 1^4) + 121 (3, 2^3,1) + 129 (3, 2^2, 1^3)\\
& & \qquad\qquad\qquad\qquad \qquad\qquad\qquad\qquad\qquad\qquad\qquad\qquad\qquad + 64 (3,2,1^5) + 17 (3, 1^7) \\
& & +23 (2^5) + 32 (2^4, 1^2) + 32 (2^3, 1^4) + 12 (2^2, 1^8) + 4( 2, 1^{10}) + 2 (1^{10});\\
%
& \psi(S_{10},\bar A_{10}) &=20 (10)+28 (9,1) + 66 (8,2) +45 (8,1^2) +78 (7,3) + 142 (7,2,1) +74 (7, 1^3)\\
& &  + 83 (6,4) + 212 (6,3,1) + 163 (6,2^2) + 214 (6,2,1^2) +79 (6, 1^4) \\
& & +26 (5^2) + 172 (5,4,1) + 249 (5,3,2) +307 (5,3,1^2) + 273 (5,2^2,1) + 223 (5,2, 1^3) +62 (5, 1^5)\\
& & + 142 (4^2,2) + 142 (4^2, 1^2) +100 ( 4,3^2) + 360(4,3,2,1) +241 (4,3,1^3) + 153 (4,2^3)\\
&& \qquad\qquad\qquad\qquad\qquad\qquad\qquad\qquad\qquad + 246 ( 4,2^2,1^2) + 156 ( 4,2,1^4) + 36 (4, 1^6)\\
& & + 97 (3^3,1)  + 104 (3^2, 2^2)+ 194 ( 3^2,2,1^2) + 88 (3^2, 1^4) + 120 (3, 2^3,1) + 129 (3, 2^2, 1^3)\\
& & \qquad\qquad\qquad\qquad \qquad\qquad\qquad\qquad\qquad\qquad\qquad\qquad\qquad + 64 (3,2,1^5) + 16 (3, 1^7) \\
& & +23 (2^5) + 33 (2^4, 1^2) + 31 (2^3, 1^4) + 13 (2^2, 1^8) + 5( 2, 1^{10});
\end{eqnarray*}\end{tiny}
\begin{tiny}\begin{eqnarray*}
& \psi(S_{11}, A_{11}) &= 29 (11) + 41 (10,1) + 100 (9,2) + 73 (9,1^2) + 142 (8,3) + 248 (8,2,1) + 125 (8, 1^3)\\
&& + 164 (7,4) + 432 (7,3,1) + 314 (7,2^2) + 426 (7,2,1^2) + 151 (7,1^4)\\
&& +107 (6,5) + 465 (6,4,1) + 621 (6,3,2) + 738 (6,3,1^2) + 645 (6,2^2,1) + 519 (6,2,1^3) + 140 (6, 1^5) \\
&&+ 200 (5^2,1) + 566 (5,4,2) + 619 (5,4,1^2) +349 (5,3^2) +1188 (5,3,2,1) + 769 (5,3,1^3) + 432 (5,2^3)\\ 
&& \qquad\qquad\qquad\qquad\qquad\qquad\qquad\qquad\qquad\qquad\qquad\qquad + 742 (5,2^2,1^2) +441 (5,2,1^4) + 99 (5,1^6)  \\
&&+ 239 (4^2,3) + 644 (4^2,2,1) +384 (4^2,1^3) + 560 (4,3^2,1) + 614 (4,3,2^2) + 1043 (4,3,2,1^2)\\
&&\quad\qquad\qquad\qquad\qquad\qquad+484 (4,3,1^4) + 515 (4,2^3,1) +  533 (4,2^2,1^3) + 254 (4,2,1^5) + 53 (4,1^7)    \\
&&+ 202(3^3,2) + 287 (3^3,1^2) +420 (3^2,2^2,1) + 410 ( 3^2,2,1^3) + 155 (3^2,1^5) +146 (3,2^4) +282 (3,2^3,1^2)
\\&&\qquad\qquad\qquad\qquad\qquad\qquad\qquad\qquad\qquad\qquad\qquad\qquad  +223 (3,2^2,1^4) + 93 (3,2,1^6) + 21 (3,1^8) \\
&& + 53 (2^5,1) +66 (2^4, 1^3) + 44 (2^3, 1^5) +17 (2^2, 1^7) +5 (2, 1^9) + 2 (1^{11}) \\
& \psi(S_{11},\bar A_{11}) &= 27 (11) + 42 (10,1) + 101 (9,2) + 72 (9,1^2) + 141 (8,3) + 248 (8,2,1) + 125 (8, 1^3)\\
&& + 164 (7,4) + 432 (7,3,1) + 314 (7,2^2) + 427 (7,2,1^2) + 150 (7,1^4)\\
&& +108 (6,5) + 465 (6,4,1) + 621 (6,3,2) + 738 (6,3,1^2) + 644 (6,2^2,1) + 519 (6,2,1^3) + 141 (6, 1^5) \\
&&+ 199 (5^2,1) + 566 (5,4,2) + 619 (5,4,1^2) +350 (5,3^2) +1188 (5,3,2,1) + 769 (5,3,1^3) + 433 (5,2^3)\\   
&& \qquad\qquad\qquad\qquad\qquad\qquad\qquad\qquad\qquad\qquad\qquad\qquad + 742 (5,2^2,1^2) +441 (5,2,1^4) + 98 (5,1^6)  \\
&&+ 239 (4^2,3) + 644 (4^2,2,1) +385 (4^2,1^3) + 559 (4,3^2,1) + 614 (4,3,2^2) + 1043 (4,3,2,1^2)\\
&&\quad\qquad\qquad\qquad\qquad\qquad+483 (4,3,1^4) + 515 (4,2^3,1) +  533 (4,2^2,1^3) + 255 (4,2,1^5) + 53 (4,1^7)    \\
&&+ 202 (3^3,2) + 288 (3^3,1^2) +420 (3^2,2^2,1) + 410 ( 3^2,2,1^3) + 155 (3^2,1^5) +145 (3,2^4) +282 (3,2^3,1^2)
\\&&\qquad\qquad\qquad\qquad\qquad\qquad\qquad\qquad\qquad\qquad\qquad\qquad  +223 (3,2^2,1^4) + 93 (3,2,1^6) + 20 (3,1^8) \\
&& + 54 (2^5,1) +66 (2^4, 1^3) + 43 (2^3, 1^5) +18 (2^2, 1^7) +6 (2, 1^9) \\\
%
%
& \psi(S_{12},A_{12}) &= 40 (12) +58 (11,1) +155 (10,2) +112 (10,1^2) + 236 (9,3) +414(9,2,1) + 207 (9, 1^3) \\
&& + 318 (8,4) +815(8,3,1) +586 (8,2^2) +789 ( 8,2,1^2) +276 (8, 1^4)\\
&& + 260 (7,5) + 1060 (7,4,1)
 + 1355 (7,3,2) +1616 (7,3^,1^2) +1373 (7,2^2,1) + 1099 (7,2,1^3) + 288 (7, 1^5) \\
&& +124 (6^2) +767 (6,5,1) +1679 (6,4,2) + 1801 (6,4,1^2) + 960 (6,3^3) + 3187 (6,3,2,1) + 2032 (6,3,1^3) \\
&&\qquad\qquad\qquad\qquad\qquad\qquad\qquad\qquad\qquad+ 1108 (6,2^3) + 1891 (6,2^2,1^2) + 1106 (6,2,1^4) + 236 (6,1^6)\\
&& +747 (5^2,2) + 842 (5^2,1^2) +1147 (5,4,3) + 3016 (5,4,2,1) + 1771 (5,4,1^3)\\ 
&&\qquad\qquad\qquad\qquad\qquad\qquad\qquad + 2115 (5,3^2,1) + 2211 (5,3,2^2)
 + 3753 ( 5,3,2,1^2) + 1678 (5,3,1^4) \\
&& \qquad\qquad\qquad\qquad\qquad\qquad\qquad+ 1682 (5,2^3,1) + 1727 (5,2^2,1^3) +793(5,2,1^5) + 156 (5, 1^7) \\
&& + 263 (4^3) + 1426 (4^2,3,1) + 1294 (4^2,2^2) + 2053 (4^2,2,1^2) + 888 (4^2,1^4) + 1363 (4,3^2,2)\\
&& \qquad\qquad\qquad\qquad\qquad\qquad
+1901 (4,3^2,1^2)+ 2592 ( 4,3,2^2,1) + 2466 (4,3,2,1^3) + 892 (4,3,1^5)\\
&&\qquad\qquad\qquad\qquad\qquad\qquad + 684 (4,2^4) +1314 (4,2^3,1^2) + 1019 (4,2^2,1^4) + 394 (4,2,1^6) + 75 (4,1^8)\\
&& + 218 (3^4) + 906 (3^3,2,1) + 690 (3^3,1^3) + 549 (3^2,2^3) 
+ 1132 ( 3^2, 2^2, 1^2) +779 (3^2, 2, 1^4) + 258 (3^2, 1^6)\\ 
&&\qquad\qquad\qquad\qquad\qquad\qquad+ 474 (3, 2^4,1) + 574 (3, 2^3, 1^3) + 357 (3, 2^2, 1^5) + 130( 3,2,1^7) + 26 (3, 1^9) \\
&& + 66 (2^6) + 114 (2^5,1^2) +112 (2^4, 1^4) + 61 (2^3, 1^6) + 24 ( 2^2, 1^8) + 5 (2, 1^{10}) +3 (1^{12}).\\
& \psi(S_{12},\bar A_{12}) &=
37 (12) +60 (11,1) +156 (10,2) +110 (10,1^2) + 235 (9,3) +414 (9,2,1) + 208 (9, 1^3) \\
&& + 319 (8,4) +815(8,3,1) +585 (8,2^2) +790 ( 8,2,1^2) +275 (8, 1^4)\\
&& + 259 (7,5) + 1060 (7,4,1)
 + 1356 (7,3,2) +1615 (7,3^,1^2) +1373 (7,2^2,1) + 1099 (7,2,1^3) + 288 (7, 1^5) \\
&& +125 (6^2) +767 (6,5,1) +1678 (6,4,2) + 1802 (6,4,1^2) + 960 (6,3^3) + 3187 (6,3,2,1) + 2032 (6,3,1^3) \\
&&\qquad\qquad\qquad\qquad\qquad\qquad\qquad\qquad\qquad+ 1109 (6,2^3) + 1890 (6,2^2,1^2) + 1107 (6,2,1^4) + 236 (6,1^6)\\
&& +748 (5^2,2) + 841 (5^2,1^2) +1147 (5,4,3) + 3016 (5,4,2,1) + 1771 (5,4,1^3)\\ 
&&\qquad\qquad\qquad\qquad\qquad\qquad\qquad + 2115 (5,3^2,1) + 2210 (5,3,2^2)
 + 3754 ( 5,3,2,1^2) + 1677 (5,3,1^4) \\
&& \qquad\qquad\qquad\qquad\qquad\qquad\qquad+ 1682 (5,2^3,1) + 1727 (5,2^2,1^3) +793(5,2,1^5) + 155 (5, 1^7) \\
&& + 263 (4^3) + 1426 (4^2,3,1) + 1295 (4^2,2^2) + 2052 (4^2,2,1^2) + 889 (4^2,1^4) + 1363 (4,3^2,2)\\
&& \qquad\qquad\qquad\qquad\qquad\qquad
+1901 (4,3^2,1^2)+ 2592 ( 4,3,2^2,1) + 2466 (4,3,2,1^3) + 892 (4,3,1^5)\\
&&\qquad\qquad\qquad\qquad\qquad\qquad + 683 (4,2^4) +1315 (4,2^3,1^2) + 1018 (4,2^2,1^4) + 395 (4,2,1^6) + 76 (4,1^8)\\
&& + 218 (3^4) + 906 (3^3,2,1) + 690 (3^3,1^3) + 550 (3^2,2^3) 
+ 1131 ( 3^2, 2^2, 1^2) +780 (3^2, 2, 1^4) + 257 (3^2, 1^6)\\ 
&&\qquad\qquad\qquad\qquad\qquad\qquad+ 474 (3, 2^4,1) + 574 (3, 2^3, 1^3) + 357 (3, 2^2, 1^5) + 130( 3,2,1^7) + 24 (3, 1^9) \\
&& + 67 (2^6) + 113 (2^5,1^2) +113 (2^4, 1^4) + 60 (2^3, 1^6) + 25 ( 2^2, 1^8) + 7 (2, 1^{10}).
\end{eqnarray*}\end{tiny} 
%
\begin{center} Table 4: Decomposition into irreducibles of 
$\varepsilon(S_n, A_n)$ and $\varepsilon(S_n, \bar A_n)$ for $2\leq n\leq 12:$
\end{center}
\begin{tiny} \begin{eqnarray*} &\varepsilon(S_2, A_2) &= {(1^2)};\quad \varepsilon(S_2, \bar A_2)={(2)};\\
&\varepsilon(S_3, A_3) &= {(3)}+2{(1^3)}; \quad \varepsilon(S_3, \bar A_3)={(3)}+{(2,1)}\\
&\varepsilon(S_4, A_4) &=    {(4)}+ 2 {(3,1)}+\phantom{(2^2)+}  {(2,1^2)}+ 2 {(1^4)},\\  
&\varepsilon(S_4, \bar A_4) &=   {(4)}+  {(3,1)} +  {(2^2)}+ 2  {(2,1^2)};\\
&\varepsilon(S_5, A_5) &=    {(5)}+2  {(4,1)}+ 2 {(3,2)}+4 {(3,1^2)} + 2  {(2^2,1)} +  {(2,1^3)}+  3 {(1^5)},\\ 
&\varepsilon(S_5, \bar A_5) &=  2   {(5)}+2  {(4,1)} +2 {(3,2)}+3 (3,1^2)+ 2  {(2^2,1)} +3  {(2,1^3)};\\
&\varepsilon(S_6, A_6) &=2  {(6)}+3  {(5,1)}+ 4  {(4,2)}+6 {(4,1^2)} + +2  {(3^2)}+ 7  {(3,2,1)} +6  {(3,1^3)} \\
& &+ 3 {(2^3)} + 4   {(2^2,1^2)} +2  {(2,1^4)} +   4{(1^6)},\\ 
&\varepsilon(S_6, \bar A_6) &=2  {(6)}+3  {(5,1)}+ 4  {(4,2)}+6  {(4,1^2)} + 4 {(3^2)}+ 6  {(3,2,1)} +6  {(3,1^3)} \\
& & + 3 {(2^3)} + 4   {(2^2,1^2)} +4  {(2,1^4)}\\
&\varepsilon(S_7, A_7) &=2 (7) +5 (6,1) +7 (5,2) + 9(5,1^2) +7 (4,3) + 17 (4,2,1) + 11 (4,1^3) \\
&& +9 (3^2,1)+9 (3,2^2) + 17 (3,2,1^2) +  10(3,1^4)+ 7 (2^3,1)+7 (2^2,1^3)
+3 (2,1^5) + 5(1^7)\\
&\varepsilon(S_7, \bar A_7) &=3 (7) +4 (6,1) +7 (5,2) + 10(5,1^2) +7 (4,3) + 16 (4,2,1) + 12 (4,1^3) \\
&& +10 (3^2,1) +10 (3,2^2) + 16 (3,2,1^2) +  9(3,1^4)+ 7 (2^3,1)+7 (2^2,1^3)
+6 (2,1^5) \\
&\varepsilon(S_8, A_8) &= 3(8) +6(7,1)+12(6,2) +15 (6,1^2)+14  (5,3) + 32 (5,2,1) + 21 (5,1^3) +6 (4^2)\\
&&+ 34 (4,3,1) + 27 (4,2^2)+44 (4,2,1^2) + 21 (4,1^4) +18 (3^2,2) + 26 (3^2, 1^2) +33 (3,2^2,1)
\\
&& +33 (3,2,1^3) + 15 (3,1^5)  + 6 (2^4) + 15 (2^3,1^2) +11 (2^2, 1^4) + 5 (2,1^6) + 6 (1^8)\\
&\varepsilon(S_8, \bar A_8) &= 3(8) +7 (7,1)+11 (6,2) +14 (6,1^2)+15  (5,3) + 33 (5,2,1) + 21 (5,1^3) +7 (4^2)\\
&&+ 32 (4,3,1) + 26 (4,2^2)+ 45 (4,2,1^2) + 21 (4,1^4)
+ 19 (3^2,2) + 27 (3^2, 1^2) +33 (3,2^2,1)\\
&& +32 (3,2,1^3) + 14 (3,1^5)  + 7 (2^4) + 14 (2^3,1^2) +12 (2^2, 1^4) + 8 (2,1^6) \\
& \varepsilon(S_9, A_9)&=4 (9) + 9 (8,1) + 18 (7,2) + 21 (7,1^2) +28 (6,3) + 57 (6,2,1) + 37 (6, 1^3) \\
&& + 19 (5,4) + 80 (5,3,1) + 60 (5,2^2) + 98 (5,2,1^2) +41 ( 5,1^4)\\
&& +41 (4^2,1) + 77 (4,3,2) + 104 (4,3,1^2) + 104 (4,2^2,1) +98 (4,2,1^3) + 36 (4, 1^5) \\
&& + 20 (3^3) + 77 (3^2,2,1)+ 41 (3,2^3) +60 ( 3^2, 1^3)  + 80 (3,2^2,1^2) + 57 (3,2,1^4) + 23 (3,1^6) \\
&& + 19 (2^4,1) + 28 (2^3, 1^3) + 18 (2^2, 1^5) + 6 (2, 1^7) + 8 (1^9)\\
&\varepsilon(S_9, \bar A_9)&=4 (9) + 8 (8,1) + 18 (7,2) + 23 (7,1^2) +27 (6,3) + 57 (6,2,1) + 35 (6, 1^3) \\
&& + 20 (5,4) + 80 (5,3,1) + 61 (5,2^2) + 98 (5,2,1^2) +42 ( 5,1^4)\\
&& +41 (4^2,1) + 76 (4,3,2) + 104 (4,3,1^2) + 104 (4,2^2,1) +98 (4,2,1^3) + 36 (4, 1^5) \\
&& + 23 (3^3) + 76 (3^2,2,1)+ 41 (3,2^3) +61 ( 3^2, 1^3)  + 80 (3,2^2,1^2) + 57 (3,2,1^4) + 21 (3,1^6) \\
&& + 20 (2^4,1) + 27 (2^3, 1^3) + 18 (2^2, 1^5) + 11 (2, 1^7) \\
& \varepsilon(S_{10},A_{10}) &=5 (10)+11 (9,1) + 27 (8,2) +32 (8,1^2) +46 (7,3) + 96 (7,2,1) +57 (7, 1^3)\\
& &  + 47 (6,4) + 164 (6,3,1) + 118 (6,2^2) + 192 (6,2,1^2) +78 (6, 1^4) \\
& & +19 (5^2) + 141 (5,4,1) + 217 (5,3,2) +282 (5,3,1^2) + 262 (5,2^2,1) + 237 (5,2, 1^3) +77 (5, 1^5)\\
& & + 118 (4^2,2) + 144 (4^2, 1^2) +98 ( 4,3^2) + 361(4,3,2,1) +263 (4,3,1^3) + 144 (4,2^3)\\
&& \qquad\qquad\qquad\qquad\qquad\qquad\qquad\qquad\qquad + 283 ( 4,2^2,1^2) + 191 ( 4,2,1^4) + 58 (4, 1^6)\\
& & + 97 (3^3,1)  + 119 (3^2, 2^2)+ 217 ( 3^2,2,1^2) + 117 (3^2, 1^4) + 140 (3, 2^3,1) + 164 (3, 2^2, 1^3)\\
& & \qquad\qquad\qquad\qquad \qquad\qquad\qquad\qquad\qquad\qquad\qquad\qquad\qquad + 97 (3,2,1^5) + 32 (3, 1^7) \\
& & +21 (2^5) + 46 (2^4, 1^2) + 47 (2^3, 1^4) + 26 (2^2, 1^6) + 9( 2, 1^{8}) + 10 (1^{10});\\
& \varepsilon(S_{10},\bar A_{10}) &=5 (10)+12 (9,1) + 26 (8,2) +31 (8,1^2) +46 (7,3) + 97 (7,2,1) +58 (7, 1^3)\\
& &  + 46 (6,4) + 165 (6,3,1) + 118 (6,2^2) + 190 (6,2,1^2) +78 (6, 1^4) \\
& & +22 (5^2) + 139 (5,4,1) + 216 (5,3,2) +284 (5,3,1^2) + 263 (5,2^2,1) + 236 (5,2, 1^3) +79 (5, 1^5)\\
& & + 119 (4^2,2) + 145 (4^2, 1^2) +98 ( 4,3^2) + 360(4,3,2,1) +262 (4,3,1^3) + 145 (4,2^3)\\
&& \qquad\qquad\qquad\qquad\qquad\qquad\qquad\qquad\qquad + 283 ( 4,2^2,1^2) + 191 ( 4,2,1^4) + 57 (4, 1^6)\\
& & + 99 (3^3,1)  + 118 (3^2, 2^2)+ 216 ( 3^2,2,1^2) + 119 (3^2, 1^4) + 140 (3, 2^3,1) + 165 (3, 2^2, 1^3)\\
& & \qquad\qquad\qquad\qquad \qquad\qquad\qquad\qquad\qquad\qquad\qquad\qquad\qquad + 96 (3,2,1^5) + 31 (3, 1^7) \\
& & +20 (2^5) + 47 (2^4, 1^2) + 45 (2^3, 1^4) + 27 (2^2, 1^6) + 14( 2, 1^{8});
\end{eqnarray*}\end{tiny}
\begin{tiny}\begin{eqnarray*}
& \varepsilon(S_{11}, A_{11}) &= 6 (11) + 16 (10,1) + 38 (9,2) + 43 (9,1^2) + 73 (8,3) + 154 (8,2,1) + 90 (8, 1^3)\\
&& + 92 (7,4) + 307 (7,3,1) + 215 (7,2^2) + 345 (7,2,1^2) + 134 (7,1^4)\\
&& +67 (6,5) + 352 (6,4,1) + 498 (6,3,2) + 638 (6,3,1^2) + 572 (6,2^2,1) + 506 (6,2,1^3) + 157 (6, 1^5) \\
&&+ 159 (5^2,1) + 476 (5,4,2) + 568 (5,4,1^2) +315 (5,3^2) +1122 (5,3,2,1) + 782 (5,3,1^3) + 407 (5,2^3)\\ 
&& \qquad\qquad\qquad\qquad\qquad\qquad\qquad\qquad\qquad\qquad\qquad\qquad + 782 (5,2^2,1^2) +505 (5,2,1^4) + 135 (5,1^6)  \\
&&+ 215 (4^2,3) + 623 (4^2,2,1) +407 (4^2,1^3) + 558 (4,3^2,1) + 624 (4,3,2^2) + 1122 (4,3,2,1^2)\\
&&\quad\qquad\qquad\qquad\qquad\qquad+572 (4,3,1^4) + 567 (4,2^3,1) +  639 (4,2^2,1^3) + 345 (4,2,1^5) + 89 (4,1^7)    \\
&&+ 215(3^3,2) + 314 (3^3,1^2) +476 (3^2,2^2,1) + 499 ( 3^2,2,1^3) + 214 (3^2,1^5) +160 (3,2^4) +351 (3,2^3,1^2)
\\&&\qquad\qquad\qquad\qquad\qquad\qquad\qquad\qquad\qquad\qquad\qquad\qquad  +307 (3,2^2,1^4) + 154 (3,2,1^6) + 45 (3,1^8) \\
&& + 68 (2^5,1) +91 (2^4, 1^3) + 74 (2^3, 1^5) +38 (2^2, 1^7) +12 (2, 1^9) + 12 (1^{11}) \\
& \varepsilon(S_{11}, \bar A_{11}) &= 6 (11) + 15 (10,1) + 38 (9,2) + 45 (9,1^2) + 74 (8,3) + 153 (8,2,1) + 88 (8, 1^3)\\
&& + 91 (7,4) + 307 (7,3,1) + 216 (7,2^2) + 345 (7,2,1^2) + 137 (7,1^4)\\
&& +68 (6,5) + 351 (6,4,1) + 498 (6,3,2) + 640 (6,3,1^2) + 571 (6,2^2,1) + 504 (6,2,1^3) + 156 (6, 1^5) \\
&&+ 161 (5^2,1) + 475 (5,4,2) + 566 (5,4,1^2) +316 (5,3^2) +1122 (5,3,2,1) + 783 (5,3,1^3) + 407 (5,2^3)\\ 
&& \qquad\qquad\qquad\qquad\qquad\qquad\qquad\qquad\qquad\qquad\qquad\qquad + 783 (5,2^2,1^2) +505 (5,2,1^4) + 136 (5,1^6)  \\
&&+ 216 (4^2,3) + 625 (4^2,2,1) +407 (4^2,1^3) + 555 (4,3^2,1) + 624 (4,3,2^2) + 1122 (4,3,2,1^2)\\
&&\quad\qquad\qquad\qquad\qquad\qquad+571 (4,3,1^4) + 567 (4,2^3,1) +  639 (4,2^2,1^3) + 345 (4,2,1^5) + 89 (4,1^7)    \\
&&+ 216 (3^3,2) + 317 (3^3,1^2) +475 (3^2,2^2,1) + 497 ( 3^2,2,1^3) + 217 (3^2,1^5) +160 (3,2^4) +352 (3,2^3,1^2)
\\&&\qquad\qquad\qquad\qquad\qquad\qquad\qquad\qquad\qquad\qquad\qquad\qquad  +307 (3,2^2,1^4) + 153 (3,2,1^6) + 43 (3,1^8) \\
&& + 67 (2^5,1) +92 (2^4, 1^3) + 73 (2^3, 1^5) +38 (2^2, 1^7) +19 (2, 1^9)  \\
& \varepsilon(S_{12},A_{12}) &= 7 (12) +20 (11,1) +54 (10,2) +61 (10,1^2) + 114 (9,3) +233 (9,2,1) + 133 (9, 1^3) \\
&& + 162 (8,4) +534(8,3,1) +371 (8,2^2) +588 ( 8,2,1^2) +225 (8, 1^4)\\
&& + 160 (7,5) + 752 (7,4,1)
 + 1018 (7,3,2) +1292 (7,3^,1^2) +1132 (7,2^2,1) + 984 (7,2,1^3) + 292 (7, 1^5) \\
&& +70 (6^2) +576 (6,5,1) +1326 (6,4,2) + 1557 (6,4,1^2) + 820 (6,3^3) + 2826 (6,3,2,1) + 1930 (6,3,1^3) \\
&&\qquad\qquad\qquad\qquad\qquad\qquad\qquad\qquad\qquad+ 982 (6,2^3) + 1860 (6,2^2,1^2) + 1164 (6,2,1^4) + 295 (6,1^6)\\
&& +638 (5^2,2) + 727 (5^2,1^2) +999 (5,4,3) + 2793 (5,4,2,1) + 1753 (5,4,1^3)\\ 
&&\qquad\qquad\qquad\qquad\qquad\qquad\qquad + 1994 (5,3^2,1) + 2153 (5,3,2^2)
 + 3809 ( 5,3,2,1^2) + 1859 (5,3,1^4) \\
&& \qquad\qquad\qquad\qquad\qquad\qquad\qquad+ 1753 (5,2^3,1) + 1929 (5,2^2,1^3) +985(5,2,1^5) + 222 (5, 1^7) \\
&& + 223 (4^3) + 1396 (4^2,3,1) + 1243 (4^2,2^2) + 2153 (4^2,2,1^2) + 982 (4^2,1^4) + 1397 (4,3^2,2)\\
&& \qquad\qquad\qquad\qquad\qquad\qquad
+1994 (4,3^2,1^2)+ 2792 ( 4,3,2^2,1) + 2827 (4,3,2,1^3) + 1132 (4,3,1^5)\\
&&\qquad\qquad\qquad\qquad\qquad\qquad + 728 (4,2^4) +1557 (4,2^3,1^2) + 1292 (4,2^2,1^4) + 588 (4,2,1^6) + 134 (4,1^8)\\
&& + 220 (3^4) + 1001 (3^3,2,1) + 818 (3^3,1^3) + 637 (3^2,2^3) 
+ 1326 ( 3^2, 2^2, 1^2) +1019 (3^2, 2, 1^4) + 370 (3^2, 1^6)\\ 
&&\qquad\qquad\qquad\qquad\qquad\qquad+ 576 (3, 2^4,1) + 752 (3, 2^3, 1^3) + 533 (3, 2^2, 1^5) + 234( 3,2,1^7) + 62 (3, 1^9) \\
&& + 70 (2^6) + 160 (2^5,1^2) +162 (2^4, 1^4) + 115 (2^3, 1^6) + 53 ( 2^2, 1^8) + 16 (2, 1^{10}) +15 (1^{12}).\\
& \varepsilon(S_{12},\bar A_{12}) &=8 (12) +20 (11,1) +53 (10,2) +61 (10,1^2) + 114 (9,3) +233 (9,2,1) + 135 (9, 1^3) \\
&& + 163 (8,4) +534(8,3,1) +372 (8,2^2) +587 ( 8,2,1^2) +222 (8, 1^4)\\
&& + 159 (7,5) + 752 (7,4,1)
 + 1018 (7,3,2) +1291 (7,3^,1^2) +1132 (7,2^2,1) + 986 (7,2,1^3) + 295 (7, 1^5) \\
&& +70 (6^2) +576 (6,5,1) +1325 (6,4,2) + 1559 (6,4,1^2) + 819 (6,3^3) + 2827 (6,3,2,1) + 1929 (6,3,1^3) \\
&&\qquad\qquad\qquad\qquad\qquad\qquad\qquad\qquad\qquad+ 982 (6,2^3) + 1859 (6,2^2,1^2) + 1163 (6,2,1^4) + 292 (6,1^6)\\
&& +638 (5^2,2) + 728 (5^2,1^2) +1002 (5,4,3) + 2790 (5,4,2,1) + 1752 (5,4,1^3)\\ 
&&\qquad\qquad\qquad\qquad\qquad\qquad\qquad + 1994 (5,3^2,1) + 2154 (5,3,2^2)
 + 3810 ( 5,3,2,1^2) + 1860 (5,3,1^4) \\
&& \qquad\qquad\qquad\qquad\qquad\qquad\qquad+ 1752 (5,2^3,1) + 1930 (5,2^2,1^3) +985(5,2,1^5) + 225 (5, 1^7) \\
&& + 220 (4^3) + 1396 (4^2,3,1) + 1246 (4^2,2^2) + 2154 (4^2,2,1^2) + 982 (4^2,1^4) + 1395 (4,3^2,2)\\
&& \qquad\qquad\qquad\qquad\qquad\qquad
+1994 (4,3^2,1^2)+ 2791 ( 4,3,2^2,1) + 2826 (4,3,2,1^3) + 1132 (4,3,1^5)\\
&&\qquad\qquad\qquad\qquad\qquad\qquad + 727 (4,2^4) +1559 (4,2^3,1^2) + 1291 (4,2^2,1^4) + 587 (4,2,1^6) + 134 (4,1^8)\\
&& + 223 (3^4) + 1000 (3^3,2,1) + 821 (3^3,1^3) + 639 (3^2,2^3) 
+ 1325 ( 3^2, 2^2, 1^2) +1017 (3^2, 2, 1^4) + 373 (3^2, 1^6)\\ 
&&\qquad\qquad\qquad\qquad\qquad\qquad+ 576 (3, 2^4,1) + 752 (3, 2^3, 1^3) + 535 (3, 2^2, 1^5) + 232( 3,2,1^7) + 60 (3, 1^9) \\
&& + 70 (2^6) + 159 (2^5,1^2) +163 (2^4, 1^4) + 113 (2^3, 1^6) + 54 ( 2^2, 1^8) + 24 (2, 1^{10}).
\end{eqnarray*}\end{tiny}

\bibliographystyle{amsplain.bst}

\end{document}